\documentclass{amsart}
\usepackage{amsmath,amssymb,amsthm,mathtools,color}
\usepackage{graphicx} 
\usepackage{array}
\usepackage{hyperref}
\numberwithin{equation}{section}
\usepackage{multirow}
\usepackage{caption}
\usepackage{float}
\usepackage{textcomp}
\providecommand{\keywords}[1]
{
  \small	
  \textbf{\textit{Keywords---}} #1
}

\usepackage{tikz}
\usetikzlibrary{calc}
\usetikzlibrary{decorations.pathreplacing}

\theoremstyle{plain}
\newtheorem{theorem}{Theorem}[section]
\newtheorem{lemma}[theorem]{Lemma}
\newtheorem{prop}[theorem]{Proposition} 
\newtheorem{cor}[theorem]{Corollary}

\newtheorem*{conj*}{Conjecture}
\newtheorem*{theorem*}{Theorem}

\theoremstyle{definition}
 
\newtheorem{remark}[theorem]{Remark} 

\newtheorem{example}[theorem]{Example}
\newtheorem{obs}[theorem]{Observation}

\newcommand{\F}{\mathbb{F}}

\newcommand{\SL}{{\rm SL}}

\DeclareMathOperator{\Res}{Res}

\newcommand{\sgn}{{\mathrm{sgn}}}

\newcommand{\Fp}{{\mathbb{F}_p}}
\newcommand{\Fpm}{{\mathbb{F}^{\times}_p}}

\newcommand{\Mkp}{{\mathcal{M}_{\kappa}(p)}}
\newcommand{\Mzp}{{\mathcal{M}(p)}}

\newcommand{\Gkp}{{\mathcal{G}_{\kappa}(p)}}
\newcommand{\Gzp}{{\mathcal{G}(p)}}

\newcommand{\Knkpall}{{\mathcal{K}^{\mathrm{all}}_{n,\kappa}(p)}}
\newcommand{\Knkpdist}{{\mathcal{K}^{\mathrm{dist}}_{n,\kappa}(p)}}
\newcommand{\Knkpprop}{{\mathcal{K}^{\mathrm{prop}}_{n,\kappa}(p)}}

\newcommand{\Tnkpall}
{{T^{\mathrm{all}}_{n,\kappa}(p)}}
\newcommand{\Tnkpdist}
{{T^{\mathrm{dist}}_{n,\kappa}(p)}}

\newcommand{\Inkp}{{I_{n,\kappa}(p)}}

\title{
Topological properties of generalized Markoff mod $p$ graphs
}

\author[S. Satake]{Shohei Satake}
\address{Shohei Satake \\ Research and Education Institute for Semiconductors and Informatics, Kumamoto University, 
Kumamoto, 
Japan}
\email{shohei-satake@kumamoto-u.ac.jp}

\author[Y. Yamasaki]{Yoshinori Yamasaki}
\address{Yoshinori Yamasaki \\ Graduate School of Science and Engineering, Ehime University, Matsuyama, Japan}
\email{yamasaki@math.sci.ehime-u.ac.jp}


\begin{document}

\maketitle
\allowdisplaybreaks[1]

\begin{abstract} 
The generalized Markoff mod $p$ graph is defined via the equation $x^2+y^2+z^2=xyz+\kappa$ over the finite field $\mathbb{F}_p$ of prime order $p$. 
In this paper, we investigate the topological properties of the graph such as non-planarity, surface embeddability, and the existence of short cycles.
Our approach is based on a systematic construction of $K_{3,3}$-subdivisions, integrating techniques from graph theory, computer algebra, and number theory.

\end{abstract}


\section{Introduction}

The Diophantine equation
\begin{equation}
\label{def:Markoff-equation}
x^2+y^2+z^2=xyz   
\end{equation}
is called the {\it Markoff equation},
and a solution $(x,y,z)$ is called a {\it Markoff triple}.
For a prime $p>3$, 
let $\Mzp$ denote the set of all Markoff triples over $\Fp$, the finite field with $p$ elements.
The {\it Markoff mod $p$ graph} $\Gzp$ is a finite $3$-regular graph (possibly  with loops) whose vertex set is $\Mzp \setminus \{(0,0,0)\}$, and two triples are adjacent if and only if they are transformed into each other by one of the following {\it Vieta involutions}:
\[
R_1(x,y,z)=(yz-x,y,z), \quad 
R_2(x,y,z)=(x,zx-y,z), \quad 
R_3(x,y,z)=(x,y,xy-z). 
\]

The graph $\Gzp$ has been extensively  studied in the literature on ``strong approximation'' for the Markoff equation (\ref{def:Markoff-equation}). 
Indeed, Bourgain, Gamburd, and Sarnak~\cite{BGS2016} conjectured that $\Gzp$ is connected for all primes $p>3$. 
Recently Chen~\cite{C2024} made a breakthrough toward this conjecture by proving that $\Gzp$ is connected for any sufficiently large prime $p$.
A key result here is that the size of any component of $\Gzp$ must be divisible by $p$. 
Furthermore, for any $\varepsilon>0$, there are only $O(p^{\varepsilon})$ vertices that are not contained in the giant component (\cite{BGS2016}) while $\Gzp$ has approximately $p^2$ vertices (see (\ref{eq-number}) below).
Martin~\cite{Martin2025} provided an elementary proof of Chen's result concerning the  divisibility of the component sizes of $\Gzp$.
In addition, Eddy, Fuchs, Litman, Martin, and Tripeny~\cite{EFLMT2025} established an explicit lower bound for primes $p$ for which the connectivity of $\Gzp$ is guaranteed, namely, $p>3.448\times10^{392}$.
There are several algorithmic results towards the connectivity of $\Gzp$, and it was verified in \cite{B2025} that $\Gzp$ is connected for $p<10^6$.

On the other hand, de Courcy-Ireland~\cite{C2024} established the non-planarity of $\Gzp$: $\Gzp$ is planar if and only if $p=7$. This provides (indirect but positive) evidence for another well-known conjecture by Bourgain, Gamburd, and Sarnak~\cite{BGS2016, BGS2026} asserting that $\{\Gzp\}_{\text{$p$:prime}}$ is in fact an expander family, which would imply higher connectivity for $\Gzp$ (for further details on expander family, see e.g. \cite{HLW2004, L2012}).
The proof of non-planarity consists of two parts: an asymptotic evaluation of the 
Euler characteristic
of $\Gzp$ and explicit constructions of subdivisions of the complete bipartite graph $K_{3,3}$. 
The Kuratowski-Wagner theorem shows a complete characterization for planar graphs in terms of graph minors, that is, a graph is planar if and only if it contains neither $K_{3,3}$ nor the complete graph $K_5$ as a graph minor.
Since $\Gzp$ is a $3$-regular graph, $K_{3,3}$ is the unique obstruction to planarity of $\Gzp$ in the sense that 
a $3$-regular graph cannot contain a subdivision of $K_5$, and more generally, any $3$-regular graph containing $K_5$ as a minor must also contain a subdivision of $K_{3,3}$.   

There are several generalizations of \eqref{def:Markoff-equation}, such as the following equation: for a (rational) integer $\kappa$,
\begin{equation}
\label{def:Markoff-type equation}
x^2+y^2+z^2=xyz+\kappa.   
\end{equation}
Throughout this paper, we naturally treat $\kappa$ as an element of $\Fp$ when considering the solutions of \eqref{def:Markoff-type equation} modulo $p$.
For the set of all such solutions over $\Fp$, 
denoted by $\Mkp$, 
Carlitz \cite{C1954} proved that 
\begin{align}
\label{eq-number}
    \# \Mkp
=p^2+\left(3+\left(\frac{\kappa}{p}\right)\right)\left(\frac{\kappa-4}{p}\right)p+1,
\end{align}
where $(\frac{\cdot}{p})$ is the Legendre symbol.
The {\it generalized Markoff mod $p$ graph} $\Gkp$ (with respect to \eqref{def:Markoff-type equation}) is defined as a $3$-regular graph
whose vertex set is $\Mkp$ 
with the adjacency relation defined in the same manner as for $\Gzp$ (hence $\mathcal{G}_0(p)$ coincides with $\Gzp \sqcup \{(0,0,0)\}$, where $(0,0,0)$ is an isolated vertex).
The graph structure of $\Gkp$ depends sensitively on $\kappa$ and $p$ in general.
For example, in the case of $\kappa=4$ (corresponding to a degenerate surface called the {\it Cayley cubic}, cf. \cite{CL2020}), the graph $\mathcal{G}_4(p)$ possesses many small components.
While $\Gkp$ is no longer connected in general,
it is still expected that for $\kappa\neq 4$,
there exists a unique {\it giant component} $\mathcal{C}_\kappa(p)$ in  
$\Gkp$ such that $\#(\Gkp \setminus \mathcal{C}_\kappa(p))=o_\kappa(p^2)$. 
In a recent preprint~\cite{Mart2025}, Martin announced that this in fact holds for a majority of primes.
Moreover, the main result of \cite{Mart2025} would imply that $\#(\Gkp \setminus \mathcal{C}_\kappa(p))=O(1)$, representing a significant advance in the understanding of the global structure of $\Gkp$.
However, 
the precise structure of $\Gkp$ (and $\mathcal{C}_\kappa(p)$) is still largely unexplored.

In this paper, we investigate the topological properties of the generalized Markoff mod $p$ graph $\Gkp$ by combining perspectives of graph theory, computer algebra, and number theory.
In particular, by leveraging computational algebraic methods including resultants and Gr\"obner bases, in conjunction with the Chebotarev density theorem, we derive the following results based on graph minor theory.
Throughout this paper, the term {\it $K_{3,3}$-subdivision} refers to a subdivision of the complete bipartite graph $K_{3,3}$. 


\begin{theorem}
\label{thm:main}
Let $\kappa \in\mathbb{Z}\setminus \{4\}$.
Then the followings hold:
\begin{enumerate}
\item[$(1)$]
$(\mathrm{Theorems}~\ref{thm-n=1},\ \ref{thm:n=2 non-planarity} \  \text{and}\ \ref{thm-n=(p-1)/2})$ 
For each such $\kappa$ and for infinitely many primes $p$ (depending on $\kappa$) of natural density at least $1/2$, 
there exists an explicit $K_{3,3}$-subdivision in $\Gkp$.

\item[$(2)$] 
$(\mathrm{Theorem}~\ref{thm:np density})$
Moreover, for almost all $\kappa$, 
the natural density of primes $p$ for which $(1)$ holds is at least $13/16$.

    
\end{enumerate}
\end{theorem}

Theorem~\ref{thm:main} yields several topological properties of $\Gkp$.
In particular, the following corollary is an immediate consequence of Theorem~\ref{thm:main}.

\begin{cor} 
\label{cor:direct}
Under the same assumptions on $\kappa$ and $p$ of Theorem~\ref{thm:main} $(1)$,
the followings hold:
\begin{enumerate}
\item[$(1)$] 
$(\mathrm{Corollary}~\ref{cor:sufficient condition non-planarity})$
$\Gkp$ is non-planar. 
\item[$(2)$]
$(\mathrm{Corollary}~\ref{cor:cycles})$
$\Gkp$ has cycles (i.e. closed paths with no repeated vertices or edges) of lengths $6,9,10,15$, and $18$, 
all of which can be constructed explicitly.  
\end{enumerate}
\end{cor}
Note that it would be interesting to obtain results analogous to Theorem~\ref{thm:main} and Corollary~\ref{cor:direct} for other generalizations of the Markoff mod $p$ graph $\Gzp$ (e.g. \cite{CLM2025}); these will be discussed in our subsequent works.

Recall that $\Gkp$ admits the following non-trivial graph automorphisms, known as {\it double sign changes}:
\[
(x,y,z) \mapsto (x,-y,-z), \quad 
(x,y,z) \mapsto (-x,y,-z), \quad \text{and} \quad 
(x,y,z) \mapsto (-x,-y,z).
\]
These automorphisms, together with the identity, form a group isomorphic to $(\mathbb{Z}/2\mathbb{Z})^2$.
Combining this symmetry of $\Gkp$ 
with Theorem~\ref{thm:main},
we obtain an explicit construction of mutually vertex-disjoint $K_{3,3}$-subdivisions.


\begin{theorem}[Theorem~\ref{thm-disjoint-k}]
\label{thm:disjoint-intro}
For each $\kappa \in \mathbb{Z} \setminus \{ 4\}$, $\Gkp$ contains at least four mutually vertex-disjoint $K_{3,3}$-subdivisions for infinitely many primes $p$ (depending on $\kappa$) of natural density at least $1/4$ if $\kappa=2$ and $1/2$ otherwise.
Moreover, for almost all $\kappa$, the natural density of primes $p$ for which this holds is at least $5/8$.
\end{theorem}

Theorem~\ref{thm:disjoint-intro} strengthens Corollary~\ref{cor:direct} (1). Consider an embedding of $\Gkp$ into a closed surface (i.e. a compact topological surface without boundary); see e.g.~\cite{GT1987, MT2001} for the details on surface embeddings of graphs.
Recall that the Euler characteristic $\chi$ of $\Gkp$ is defined by $V-E+F$, where $V, E$, and $F$
denote the numbers of vertices, edges and faces, respectively.
Combining \cite[Theorem~1.1]{Mart2025} with Theorem~\ref{thm:disjoint-intro}, 
it follows that $\chi\leq -6$
for $\kappa$ and $p$ as in Theorem~\ref{thm:disjoint-intro},
whereas Euler's formula requires  $\chi=2$ if $\Gkp$ is planar.
Moreover, the existence of mutually vertex-disjoint $K_{3,3}$-subdivisions highlights the robustness of the non-planarity of $\Gkp$ in the sense that, by Theorem~\ref{thm:disjoint-intro}, one must remove at least four vertices to make the graph planar.
It is worth noting that, in light of \cite[Theorem~1.1]{Mart2025}, our construction of $K_{3,3}$-subdivisions implies that the Euler characteristic of the giant component $\mathcal{C}_\kappa(p)$ is at most $-6$ for any 
$\kappa\neq 4$
and infinitely many primes $p$ (see Section~\ref{sect:disjoint} for details).
From these structural perspectives, our results 
serve as heuristic evidence for the expander property of the giant component $\mathcal{C}_\kappa(p)$.
At least, 
$\mathcal{C}_\kappa(p)$ in fact possesses certain complex structure similar to that of $\Gzp$.

It is also worth mentioning that in cryptography and number theory, studying the structure of $\Gkp$ is crucial for understanding the security of the associated cryptographic hash functions proposed in \cite{FLLT2021, Sil2025}, which are expected to be expander hash functions~\cite{CLG2009}.
In particular, it is a challenging open problem to investigate the existence and distribution of short cycles in $\Gkp$, as these properties directly affect the collision resistance of the hash functions (\cite{D2023, FLLT2021}).
Corollary~\ref{cor:direct} (2) provides a modest step toward addressing this problem.

In addition, for $\kappa=0$, 
we give a complete characterization for primes $p$ for which the graph $\Gzp$ admits an embedding into closed surfaces of minimal topological complexity beyond the plane. 

\begin{theorem}[Theorem~\ref{thm-proj}]
\label{thm:embedded-intro}
Let $p>3$ be a prime. Then the graph $\Gzp$ admits an embedding into the torus 
(i.e. the orientable closed surface of Euler characteristic zero) 
or the (real) projective plane 
(i.e. the non-orientable closed surface of Euler characteristic one) 
if and only if $p=7$. 
\end{theorem}

This result is achieved through Theorem~\ref{thm:main} by finding suitable forbidden minors for these embeddings, 
as well as the numerical verification of the connectivity of $\Gzp$ for $p<10^6$ (\cite{B2025}).
Note that the graph $\mathcal{G}(7)$ is planar, and hence one can draw it on any closed surface without edge crossings. 
Consequently, Theorem~\ref{thm:embedded-intro} immediately implies the non-planarity of $\Gzp$ for any $p\neq 7$, and hence strengthens the main result of \cite{C2024}.

The rest of this paper is organized as follows. 
Section~\ref{sect:chebyshev} introduces the basic and relevant properties of  Chebyshev polynomials, which play a key role in constructing $K_{3, 3}$-subdivisions in $\Gkp$.
In Section~\ref{sec:Six triples}, we present a systematic construction of $K_{3,3}$-subdivisions based on careful computations of Chebyshev polynomials and their associated resultants.
Briefly, we first reduce the construction to the problem of solving systems of algebraic equations via Gr\"obner bases to obtain candidates for the sextuples of triples that form the desired subdivisions.  
Then we establish sufficient conditions for the ``properness'', which guarantees the claimed topological properties, of the subdivisions corresponding to the obtained sextuples.
Based on these arguments, Section~\ref{sect:non-planarity} provides explicit $K_{3,3}$-subdivisions.
Finally, Sections~\ref{sect:disjoint} and \ref{sect:embedded} are devoted to the proofs of Theorems~\ref{thm:disjoint-intro} and \ref{thm:embedded-intro}, respectively.

\section{Chebyshev polynomials}
\label{sect:chebyshev}

In this section, 
we recall some basic properties of the Chebyshev polynomials and introduce a key polynomial $A_m(x)$ that plays a central role in our study.

The Chebyshev polynomials  $T_m(x),U_m(x)\in\mathbb{Z}[x]$ of the first and second kind, respectively, are defined by the following recurrence relations:
\begin{alignat}{3}
\label{for:recursion for T_m}
 T_0(x)&= 1,  &\quad 
 T_1(x)&= x,  &\quad 
 T_{m+1}(x) 
 &= 2xT_m(x)-T_{m-1}(x) \quad (m\ge 1), \\
\label{for:recursion for U_m}
 U_0(x)&= 1, &\quad 
 U_1(x) &= 2x, &\quad 
 U_{m+1}(x) 
 &= 2xU_m(x)-U_{m-1}(x) \quad (m\ge 1).
\end{alignat}
For convenience, we extend these definitions to negative indices by setting $T_{-1}(x)=x$, $T_{-2}(x)=2x^2-1$ and 
$U_{-1}(x)=0$, $U_{-2}(x)=-1$.
We summarize the properties used in this paper below.

\begin{lemma}
\label{lem:T_m and U_m properties} 
\begin{itemize}
\item[$(1)$]
We have 
\begin{alignat}{2}
\label{for:explicit expression of T_m}
 T_{m}(x) &= \sum^{\lfloor m/2\rfloor}_{j=0} t_{m,j}(2x)^{m-2j}, \quad & t_{m,j} &= \frac{1}{2}(-1)^j \left\{ \binom{m-j}{j} + \binom{m-j-1}{j} \right\}, \\
\label{for:explicit expression of U_m}
 U_{m}(x) &= \sum^{\lfloor m/2\rfloor}_{j=0} u_{m,j}(2x)^{m-2j}, \quad & u_{m,j} &= (-1)^j \binom{m-j}{j}.
\end{alignat}
\item[$(2)$]
$U_m(0)=(-1)^{m/2}$ if $m$ is even and $0$ otherwise.
\item[$(3)$]
$U_m(\pm 1)=(\pm 1)^{m}(m+1)$.
\item[$(4)$]
$U_{2m}(x)+1=2T_m(x)U_{m}(x)$.
\item[$(5)$]
$U_{2m}(x)-1=2T_{m+1}(x)U_{m-1}(x)$.
\item[$(6)$]
$U_{2m-1}(x)=2T_m(x)U_{m-1}(x)$.
\end{itemize}
\end{lemma}

Define the polynomial $A_m(x)\in\mathbb{Z}[x]$ of degree $m$ by
\begin{align}
\label{def:A_m}
 A_m(x)
\coloneq
\prod_{\substack{d\,|\,2m+1 \\ d\ge 3}}\Psi_d(2x),
\end{align}
where $\Psi_m(x)$ is the minimal polynomial of $2\cos(2\pi/m)$ over $\mathbb{Q}$, which was studied, for example in \cite{L1933}.
As shown in \cite[Lemma~2.2]{LW2011},
$A_m(x)$ satisfies the recurrence relation 
\begin{alignat}{3}
\label{for:recursion formula for Am}
A_0(x)&= 1, &\quad 
A_1(x)&= 2x+1, &\quad 
A_{m+1}(x) 
&= 2xA_m(x)-A_{m-1}(x) \quad (m\ge 1).
\end{alignat}
Together with \eqref{for:recursion for U_m}, this shows that
\begin{align}
\label{for:relation between A and U}
A_{m}(x)=U_m(x)+U_{m-1}(x).
\end{align}
In this paper, $A_m(x/2)$ plays an important role. 
The first few are given as follows:
\begin{align*}
 A_0\left(\frac{x}{2}\right)&=1, \quad   
 A_1\left(\frac{x}{2}\right)=x+1, \quad   
 A_2\left(\frac{x}{2}\right)=x^2+x-1, \quad   
 A_3\left(\frac{x}{2}\right)=x^3+x^2-2 x-1, \\[5pt]
 A_4\left(\frac{x}{2}\right)&=x^4+x^3-3 x^2-2 x+1, \quad 
 A_5\left(\frac{x}{2}\right)=x^5+x^4-4 x^3-3 x^2+3 x+1.
\end{align*}
We also summarize several properties of $A_m(x)$.
Throughout this paper, we denote the discriminant of a polynomial $f(x)$ by $D(f(x))$. 

\begin{lemma}
\label{lem:Am properties} 
\begin{itemize}
\item[$(1)$] 
We have 
\begin{equation}
\label{for:explicit expression of Am}
 A_{m}(x)
=\sum^{m}_{j=0}a_{m,j}(2x)^{m-j},
\quad 
a_{m,j}=(-1)^{\lfloor \frac{j}{2}\rfloor}\binom{\lfloor \frac{2m-j}{2}\rfloor}{\lfloor \frac{j}{2}\rfloor}.
\end{equation}
\item[$(2)$]
$A_m(0)=(-1)^{\lfloor\frac{m}{2}\rfloor}$.
\item[$(3)$]
$A_m(1)=2m+1$ and $A_m(-1)=(-1)^m$.
\item[$(4)$]
For $m\ge 2$,
\begin{equation}
\label{for:disc of A}
D\left(A_m(x)\right)=2^{m(m-1)}(2m+1)^{m-1}.
\end{equation}
\end{itemize}
\end{lemma}
\begin{proof}
The first three properties are obtained in 
\cite[Theorem~2.2, Corollary~3.2 
]{LW2011} (or derived via \eqref{for:relation between A and U}).
The last one follows from \cite[Theorem~4]{DS2005} by applying the expression  \eqref{for:relation between A and U}.
\end{proof}

Moreover, we have the following factorization of $U_m(x)$ in terms of $A_m(x)$.

\begin{lemma}
\label{lem:U by A}
\begin{itemize}
\item[$(1)$] 
$U_{2m}(x)=(-1)^mA_m(x)A_m(-x)$.
\item[$(2)$] 
$U_{2m-1}(x)+1=(-1)^{m-1}A_{m-1}(-x)A_{m}(x)$.
\item[$(3)$] 
$U_{2m-1}(x)-1=(-1)^{m}A_{m-1}(x)A_{m}(-x)$.
\end{itemize}
\end{lemma}
\begin{proof}
The first equation is given in \cite[(1.4)]{LW2011}.
The second and third equations are derived from (56) and (57) of \cite{K2022}, respectively, by applying the identity
\begin{align}
\label{for:A_m -x}
 (-1)^mA_m(-x)
 =\prod_{\substack{d\,|\,2m+1 \\ d\ge 3}}\Psi_{2d}(2x),
\end{align} 
which follows from \eqref{def:A_m}.
\end{proof}

\section{A construction of $K_{3,3}$-subdivisions in $\Gkp$}
\label{sec:Six triples}

\subsection{Sextuples of Markoff triples}

In this section, 
we construct a $K_{3,3}$-subdivision 
in $\Gkp$
from a sextuple 
\begin{equation}
\label{for:K}
K=\left(\begin{array}{ccc}
X_1,\!\! & X_2,\!\! & X_3  \\
Y_1,\!\! & Y_2,\!\! & Y_3
\end{array}\right)
\end{equation}
of Markoff triples  
$X_1,X_2,X_3,Y_1,Y_2,Y_3\in \Mkp$,
where  
\[
X_1=(a_1,a_2,a_3), \quad 
X_2=(b_1,b_2,b_3), \quad
X_3=(c_1,c_2,c_3)
\]
and $Y_i=R_i(X_i)$ for $1\le i\le 3$.
We begin with the following proposition.

\begin{prop}
For any $\kappa\in\mathbb{Z}$
and any prime $p>3$,
the graph $\Gkp$ contains no subgraph isomorphic to $K_{3,3}$.  
\end{prop}
\begin{proof}
Suppose that the six triples $X_1, X_2, X_3, Y_1, Y_2, Y_3\in \Mkp$ form a subgraph isomorphic to $K_{3,3}$.  
Then the only possibility is that
$Y_1=R_3(X_2)=R_2(X_3)$,
$Y_2=R_1(X_3)=R_3(X_1)$ and 
$Y_3=R_2(X_1)=R_1(X_2)$.
This implies $a_1=b_1=c_1$, $a_2=b_2=c_2$, and $a_3=b_3=c_3$,
which means $X_1=X_2=X_3$.
This contradicts the assumption that the triples form $K_{3,3}$.
\end{proof}

For $n\in\mathbb{Z}_{>0}$,
we study sextuples $K=\left(\begin{array}{ccc}
X_1,\!\! & X_2,\!\! & X_3  \\
Y_1,\!\! & Y_2,\!\! & Y_3
\end{array}\right)$ 
satisfying the following system of equations:
\begin{align}
\label{for:explicit minors equation}
\begin{split}
 Y_1&=R_1(X_1)=(R_2R_1)^n(X_2)=(R_3R_1)^n(X_3),\\   
 Y_2&=R_2(X_2)=(R_3R_2)^n(X_3)=(R_1R_2)^n(X_1),\\   
 Y_3&=R_3(X_3)=(R_1R_3)^n(X_1)=(R_2R_3)^n(X_2),
\end{split}
\end{align}
which is equivalent to
\begin{align}
\label{for:explicit minors equation 2}
 X_j=R_j(R_iR_j)^n(X_i)=R_j(R_kR_j)^n(X_k)\quad \text{for any $\{i,j,k\}=\{1,2,3\}$}.
\end{align}
One may expect that such a sextuple $K$ yields the desired subdivision
(see Figure~\ref{fig:explicit minors}).
Notice that, since $R_i$ changes only the $i$-th coordinate for $1\le i\le 3$, 
the equations in \eqref{for:explicit minors equation}
immediately imply
\begin{equation}
\label{for:easily verified} 
(b_1,c_2,a_3)=(c_1,a_2,b_3).
\end{equation}

\begin{figure}[h]
\begin{center}
\begin{tikzpicture}
\draw(0,0)--(0,3);
\draw(0,0)--(4,3);
\draw(0,0)--(8,3);
\draw(4,0)--(0,3);
\draw(4,0)--(4,3);
\draw(4,0)--(8,3);
\draw(8,0)--(0,3);
\draw(8,0)--(4,3);
\draw(8,0)--(8,3);
\draw(0,0)node[below]{$Y_1$};
\draw(4,0)node[below]{$Y_2$};
\draw(8,0)node[below]{$Y_3$};
\draw(0,3)node[above]{$X_1$};
\draw(4,3)node[above]{$X_2$};
\draw(8,3)node[above]{$X_3$};
\draw(0,2.4)node[left]{{\footnotesize $R_1$}};
\draw(4,2.4)node[left]{{\footnotesize $R_2$}};
\draw(8,2.4)node[right]{{\footnotesize $R_3$}};
\draw(1.4,1.9)node[left]{{\footnotesize $(R_1R_2)^n$}};
\draw(2,2.9)node[left]{{\footnotesize $(R_1R_3)^n$}};
\draw(4.5,2.7)node[right]{{\footnotesize $(R_2R_3)^n$}};
\draw(3.6,2.7)node[left]{{\footnotesize $(R_2R_1)^n$}};
\draw(6,2.9)node[right]{{\footnotesize $(R_3R_1)^n$}};
\draw(6.6,1.9)node[right]{{\footnotesize $(R_3R_2)^n$}};
\end{tikzpicture}
\caption{A sextuple of Markoff triples}
\label{fig:explicit minors}
\end{center}
\end{figure}
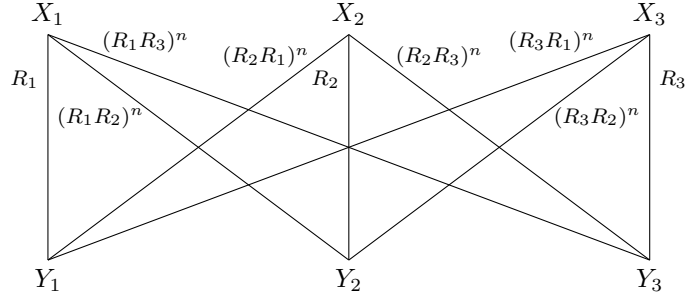

For $\kappa\in\mathbb{Z}$, 
a prime $p>3$, and $n\in\mathbb{Z}_{>0}$, 
let $\Knkpall$ denote the set of all sextuples
$K$ of the form \eqref{for:K} satisfying \eqref{for:explicit minors equation}.
From now on, we always identify $K\in \Knkpall$ 
with the subgraph of $\Gkp$ consisting of
six vertices $X_1,X_2,X_3,Y_1,Y_2,Y_3$,
three $X_i$-$Y_i$-paths $(X_i,R_i(X_i)=Y_i)$ of length $1$ for $1\le i\le 3$, and six $X_i$-$Y_j$-paths
\[
(X_i,R_j(X_i),(R_iR_j)(X_i),\ldots,R_j(R_iR_j)^{n-1}(X_i),
(R_iR_j)^n(X_i)=Y_j)
\] 
of length $2n$ for $1\le i\ne j \le 3$.
We say that $K$ is 
{\it distinct}
if $X_1,X_2,X_3,Y_1,Y_2,Y_3$ are all distinct.
Moreover, we call $K$
{\it proper} if 
none of $Z\in \{X_1,X_2,X_3,Y_1,Y_2,Y_3\}$ is an internal vertex of any $X_i$-$Y_j$-path for $1\leq i\ne  j\leq 3$.
Note that $K$ cannot be proper if $n$ is too large relative to $p$. 
In fact, $K$ is not proper if, for example, $X_i$ or $Y_j$ is fixed by $(R_iR_k)^m$ or $(R_jR_k)^m$ for some $1 \leq i, j\leq 3$, $k\neq i,j$, and $m<n$.
The conditions under which such situations occur are discussed in \cite[Proposition~7.1]{C2024}.
Put
\begin{align*}
 \Knkpdist
\coloneq\{K\in \Knkpall\,|\,\text{$K$ is distinct}\}, 
\qquad
 \Knkpprop
\coloneq\left\{K\in \Knkpdist\,|\,
\text{$K$ is proper}
\right\}.
\end{align*}

The next lemma shows that any element of $\Knkpprop$
yields a $K_{3,3}$-subdivision in $\Gkp$.


\begin{lemma}
\label{lem:properties of Kndd}
Suppose that $\Knkpprop\ne\emptyset$,
and let  
$K=\left(\begin{array}{ccc}
X_1,\!\! & X_2,\!\! & X_3  \\
Y_1,\!\! & Y_2,\!\! & Y_3
\end{array}\right)\in\Knkpprop$.
\begin{itemize}
\item[$(1)$]
Any internal vertex $W$ of the $X_i$-$Y_j$-path satisfies  $R_i(W)\ne W\ne R_j(W)$
for all $1\leq i\ne j\leq 3$.
\item[$(2)$]
Any distinct $X_i$-$Y_j$ and $X_{i'}$-$Y_{j'}$ paths share no edges for all $1 \le i,j \le 3$ and $1 \le i', j' \le 3$.
\end{itemize}
These in particular imply that 
every $K\in\Knkpprop$ is a $K_{3,3}$-subdivision in $\Gkp$.
\end{lemma}
\begin{proof}
(1) By symmetry of $\Gkp$, 
it suffices to consider the $X_1$-$Y_2$-path.
Suppose that there exists an internal vertex $W$ on this path for which the claim does not hold. 
There are two cases to consider:
$W=(R_1R_2)^m(X_1)$ for some $0<m<n$ with $R_2(W)=W$, or $W=R_2(R_1R_2)^m(X_1)$ for some $0\le m<n$ with $R_1(W)=W$.
Suppose the former holds.
Then, we have $R_2(R_1R_2)^{2m}(X_1)=X_1$,
which implies that $X_1$ (resp. $X_2$) is an internal vertex of the $X_1$-$Y_2$-path if $0<m<\frac{n}{2}$ (resp. $\frac{n}{2}\le m<n$).
This never occurs because $K\in\Knkpprop$.
The latter case is similar.

\smallbreak 

(2) By the above symmetry, 
it suffices to consider the cases $(i,j)=(1,1),(1,2)$.
The former is clear. 
For the latter, the problem reduces, 
again by symmetry, to the cases 
$(i',j')=(2,1),(1,3),(2,3)$.
\begin{itemize}
\item[(i)] 
If the $X_1$-$Y_2$-path and the $X_2$-$Y_1$-path share a common edge corresponding to $R_1$ 
(the case for $R_2$ is similar),
then we have $R_2(R_1R_2)^m(X_1)=(R_2R_1)^l(X_2)$ for some $0\le m<n$ and $0<l<n$.
This implies that 
$X_2=R_2(R_1R_2)^{m-l}(X_1)$ if $m\ge l$ and $X_1=R_1(R_2R_1)^{l-m-1}(X_2)$ otherwise,
which leads to a contradiction as above.
\item[(ii)] 
If the $X_1$-$Y_2$-path and the $X_1$-$Y_3$-path share a common edge, then they also share a common internal vertex, say $W=(x,y,z)$.
Since $R_i$ changes only the $i$-th coordinate,
we have $y=b_1$ and $z=c_1$.
Therefore, $x=a_1$ or $b_1c_1-a_1$, which implies that $W=X_1$ or $Y_1$. This is a contradiction.
\item[(iii)] 
A similar argument to (ii), using the relation  \eqref{for:easily verified},
shows that the $X_1$-$Y_2$-path and the $X_2$-$Y_3$-path share no common edges.
\end{itemize}
\end{proof}


Our target is $\Knkpprop$;
since $K_{3,3}$ is a forbidden minor for planar graphs by the Kuratowski-Wagner theorem, the following is an immediate consequence of Lemma~\ref{lem:properties of Kndd}.

\begin{cor}
\label{cor:sufficient condition non-planarity}
The graph $\Gkp$ is non-planar if there exists $n\in\mathbb{Z}_{>0}$ such that $\Knkpprop\ne\emptyset$.
\qed
\end{cor}

Moreover, the result
$\Knkpprop\ne\emptyset$ implies not only the
non-planarity of $\Gkp$
but also the existence of cycles of a specific length.

\begin{cor}
\label{cor:cycles}
The graph $\Gkp$ contains cycles 
of lengths $4n+2$, $6n+3$, and $8n+2$ if there exists $n\in\mathbb{Z}_{>0}$ such that $\Knkpprop\ne\emptyset$.
\end{cor}
\begin{proof}
For any $K\in\Knkpprop$, 
we have the following cycles of lengths:
\begin{itemize}
\item $4n+2$\,:
$X_1
\overset{1}{\to} Y_1
\overset{2n}{\to} X_2
\overset{1}{\to} Y_2
\overset{2n}{\to} X_1$, 
\item $6n+3$\,: 
$X_1
\overset{1}{\to} Y_1
\overset{2n}{\to} X_2
\overset{1}{\to} Y_2
\overset{2n}{\to} X_3
\overset{1}{\to} Y_3
\overset{2n}{\to} X_1$, 
\item $8n+2$\,:
$X_1
\overset{1}{\to} Y_1
\overset{2n}{\to} X_2
\overset{2n}{\to} Y_3
\overset{1}{\to} X_3
\overset{2n}{\to} Y_2
\overset{2n}{\to} X_1$.
\end{itemize}
\end{proof}


To study $\Knkpprop$ in detail,
we first clarify all elements of 
$\Knkpall$ and $\Knkpdist$. 
For $\alpha,\beta\in\mathbb{F}_p$, let 
\begin{equation}
\label{def:K alpha beta}
K_{(\alpha,\beta)}
\coloneq\left(\begin{array}{ccc}
(\alpha,\beta,\beta),\!\! & 
(\beta,\alpha,\beta),\!\! & 
(\beta,\beta,\alpha)  \\
(\overline{\alpha},\beta,\beta),\!\! &
(\beta,\overline{\alpha},\beta),\!\! & 
(\beta,\beta,\overline{\alpha}) 
\end{array}\right),
\end{equation}
where $\overline{\alpha}\coloneq
\beta^2-\alpha\in\mathbb{F}_p$.
Notice that $K_{(\alpha,\beta)}$ is a sextuple of Markoff triples if $f_{\kappa}(\beta,\alpha)=0$,
where
\[
f_{\kappa}(x,y)
\coloneq y^2-x^2y+2x^2-\kappa.
\]
Moreover, 
it is easily verified that $K_{(\alpha,\beta)}$ and $K_{(\overline{\alpha},\beta)}$
coincide as subgraphs of $\Gkp$.

Let $\Sigma$ be the group of double sign changes, which is isomorphic to $(\mathbb{Z}/2\mathbb{Z})^2$.
The following is direct.
\begin{lemma}
\label{lem-double-sign-k}
For any $\sigma\in \Sigma$ and $1\leq i \leq 3$,
$\sigma$ and $R_i$ commute.
In particular, $\sigma$ is an automorphism of $\Gkp$.
\end{lemma}

From this lemma,
$\Sigma$ naturally acts on
$\Knkpall$, $\Knkpdist$ and $\Knkpprop$ as 
\[
\sigma (K)
\coloneq
\left(\begin{array}{ccc}
\sigma(X_1),\!\! & \sigma(X_2),\!\! & \sigma(X_3)  \\
\sigma(Y_1),\!\! & \sigma(Y_2),\!\! & \sigma(Y_3)
\end{array}\right).
\]

By taking account of double sign changes, it is possible to enumerate all solutions of (\ref{for:explicit minors equation}).
\begin{theorem}
\label{thm:solutions}
For $\kappa\in\mathbb{Z}$, a prime $p>3$,
and $n\in\mathbb{Z}_{>0}$, define 
\begin{align*}
\Tnkpall&\coloneq
\left\{(\alpha,\beta)\in \Fp\times \F^{\times}_p\,\left|\,
A_n\left(\frac{\beta}{2}\right)=0,\right.\,
f_{\kappa}(\beta,\alpha)=0
\right\},\\
\Tnkpdist&\coloneq
\left\{(\alpha,\beta)\in \Tnkpall\,\left|\,
\beta^2(3-\beta)\ne \kappa,\,
\beta^2(8-\beta^2)\ne 4\kappa\right.
\right\}.
\end{align*}
Then, we have 
\begin{align}
\label{eq:Knk all}
 \Knkpall
&=\left\{\sigma (K_{(\alpha,\beta)})\,\left|\,(\alpha,\beta)\in \Tnkpall,\ \sigma\in \Sigma\right.\right\}\sqcup \mathcal{E}_{n,\kappa}(p),\\
\label{eq:Knk prime}
 \Knkpdist
&=\left\{\sigma (K_{(\alpha,\beta)})\,\left|\,(\alpha,\beta)\in \Tnkpdist,\ \sigma\in \Sigma\right.\right\},
\end{align}
where 
\[
\mathcal{E}_{n,\kappa}(p)
\coloneq
\begin{cases}
 \left\{\left(\begin{array}{ccc}
O,\!\! & O,\!\! & O \\
O,\!\! & O,\!\! & O
\end{array}\right)\right\} & \kappa=0,\\[15pt]
\left\{\left.\sigma\left(\begin{array}{ccc}
P\!\! & P\!\! & P \\
P\!\! & P\!\! & P
\end{array}\right)\,\right|\,\sigma\in\Sigma\right\}  & \kappa=4 \ \ \text{and} \ \ 2n+1\not\equiv 0\pmod{p},\\[15pt]
 \emptyset & \text{otherwise},
\end{cases}
\]
with $O\coloneq (0,0,0)$ and $P\coloneq (2,2,2)$.
\end{theorem}


To prove the theorem, 
we employ the matrix representation of the composition $R_iR_j$.
For any $\{i,j,k\}=\{1,2,3\}$, 
write $R_j(x_1,x_2,x_3)=(z_1,z_2,z_3)$ and 
$R_iR_j(x_1,x_2,x_3)=(y_1,y_2,y_3)$.
Then, since $R_{i}$ changes only the $i$-th coordinate of Markoff triples, we observe that 
$z_j=x_j$, $z_k=x_k$, $y_k=x_k$ and 
\begin{align}
\label{for:RiRj matrix transformation}
\left[
\begin{array}{cc}
 z_i \\
 z_j 
\end{array}
\right]
=
\left[
\begin{array}{cc}
 1 & 0 \\
 x_k & -1 
\end{array}
\right]
\left[
\begin{array}{cc}
 x_i \\
 x_j 
\end{array}
\right],
\quad 
\left[
\begin{array}{cc}
 y_i \\
 y_j 
\end{array}
\right]
=
 R(x_k)
\left[
\begin{array}{cc}
 x_i \\
 x_j 
\end{array}
\right],
\end{align}
where 
\[
R(x)
\coloneq
\left[
\begin{array}{cc}
 x^2-1 & -x \\
 x & -1 
\end{array}
\right]
=\left[
\begin{array}{cc}
 x & -1 \\
 1 & 0 
\end{array}
\right]^2
\in\SL_2(\F_p).
\]
More generally, 
for any $m\in\mathbb{Z}_{\ge 0}$, 
the matrix expressions of $(R_iR_j)^m$ and $R_j(R_iR_j)^m$ 
can be given in terms of the Chebyshev polynomials of the second kind as follows.

\begin{lemma}
\label{lem:F G}
Let $m\in\mathbb{Z}_{\ge 0}$ and $\{i,j,k\}=\{1,2,3\}$. 
For $(x_1,x_2,x_3)\in \Mkp$, write
\[
(R_iR_j)^m(x_1,x_2,x_3)
=\left(y^{(m)}_1,y^{(m)}_2,y^{(m)}_3\right), \quad 
R_j(R_iR_j)^m(x_1,x_2,x_3)
=\left(z^{(m)}_1,z^{(m)}_2,z^{(m)}_3\right).
\]
Then, we have $y^{(m)}_k=z^{(m)}_k=x_k$ and 
\begin{align*}
\left[
\begin{array}{c}
 y^{(m)}_i \\ 
 y^{(m)}_j
\end{array}
\right]
=F_m(x_k)
\left[
\begin{array}{c}
 x_i \\ 
 x_j
\end{array}
\right], \quad 
\left[
\begin{array}{c}
 z^{(m)}_i \\ 
 z^{(m)}_j
\end{array}
\right]
=G_m(x_k)
\left[
\begin{array}{c}
 x_i \\ 
 x_j
\end{array}
\right],
\end{align*}
where 
\begin{align*}
F_m(x)
\coloneq
\left[
\begin{array}{cc}
 U_{2m}\left(x/2\right) & 
 -U_{2m-1}\left(x/2\right) \\
 U_{2m-1}\left(x/2\right) & 
 -U_{2m-2}\left(x/2\right)
\end{array}
\right],\quad 
G_m(x)
\coloneq
\left[
\begin{array}{cc}
 U_{2m}\left(x/2\right) & 
 -U_{2m-1}\left(x/2\right) \\
 U_{2m+1}\left(x/2\right) & 
 -U_{2m}\left(x/2\right)
\end{array}
\right].
\end{align*}
\end{lemma}
\begin{proof}
One can prove by induction on $m$, 
in conjunction with \eqref{for:recursion for U_m}, 
that 
\[
 R(x)^m=
\left[
\begin{array}{cc}
 U_{2m}\left(x/2\right) & 
 -U_{2m-1}\left(x/2\right) \\
 U_{2m-1}\left(x/2\right) & 
 -U_{2m-2}\left(x/2\right)
\end{array}
\right].
\]
Hence the claims follows from \eqref{for:RiRj matrix transformation} by noting that  
$F_m(x)=R(x)^m$ 
and $G_m(x)=\left[
\begin{array}{cc}
 1 & 0 \\ 
 x & -1
\end{array}
\right]F_m(x)$.

\end{proof}

\begin{proof}
[Proof of Theorem~\ref{thm:solutions}]  
(1) Let $\widetilde{\mathcal{K}}^{\text{all}}_{n,\kappa}(p)$
denote the right-hand side of \eqref{eq:Knk all}.
We first show that  $\Knkpall\subset\widetilde{\mathcal{K}}^{\text{all}}_{n,\kappa}(p)$.
Take  
$K=\left(\begin{array}{ccc}
X_1,\!\! & X_2,\!\! & X_3  \\
Y_1,\!\! & Y_2,\!\! & Y_3
\end{array}\right)\in\Knkpall$.
By virtue of Lemma~\ref{lem:F G},
\eqref{for:explicit minors equation 2} can be rewritten as follows:  
\begin{alignat*}{2}
\left[
\begin{array}{c}
 a_2 \\
 a_1
\end{array}
\right]
&=
G_n(b_3)
\left[
\begin{array}{c}
 b_2 \\
 b_1
\end{array}
\right], 
\quad
&& \left[
\begin{array}{c}
 a_3 \\
 a_1
\end{array}
\right]
=
G_n(c_2)
\left[
\begin{array}{c}
 c_3 \\
 c_1
\end{array}
\right]
, \\
\left[
\begin{array}{c}
 b_3 \\
 b_2
\end{array}
\right]
&=
G_n(c_1)
\left[
\begin{array}{c}
 c_3 \\
 c_2
\end{array}
\right], 
\quad 
&& \left[
\begin{array}{c}
 b_1 \\
 b_2
\end{array}
\right]
=
G_n(a_3)
\left[
\begin{array}{c}
 a_1 \\
 a_2
\end{array}
\right]
, \\
\left[
\begin{array}{c}
 c_1 \\
 c_3
\end{array}
\right]
&=
G_n(a_2)
\left[
\begin{array}{c}
 a_1 \\
 a_3
\end{array}
\right], 
\quad 
&& \left[
\begin{array}{c}
 c_2 \\
 c_3
\end{array}
\right]
=
G_n(b_1)
\left[
\begin{array}{c}
 b_2 \\
 b_3
\end{array}
\right]
. 
\end{alignat*}
In view of \eqref{for:easily verified},
this is equivalent to 
\begin{equation}
\label{for:solutions by Chebyshev}
\begin{alignedat}{2}
\left[
\begin{array}{c}
 c_3 \\
 c_1
\end{array}
\right]
&= G'_n(a_2)
\left[
\begin{array}{c}
 a_1 \\
 a_3
\end{array}
\right],
\qquad
& G''_n(a_2)
\left[
\begin{array}{c}
 a_1 \\
 a_3
\end{array}
\right]
&=
\left[
\begin{array}{c}
 0 \\
 0
\end{array}
\right], \\
\left[
\begin{array}{c}
 a_1 \\
 a_2
\end{array}
\right]
&= G'_n(b_3)
\left[
\begin{array}{c}
 b_2 \\
 b_1
\end{array}
\right],
\qquad
& G''_n(b_3)
\left[
\begin{array}{c}
 b_2 \\
 b_1
\end{array}
\right]
&=
\left[
\begin{array}{c}
 0 \\
 0
\end{array}
\right], \\
\left[
\begin{array}{c}
 b_2 \\
 b_3
\end{array}
\right]
&= G'_n(c_1)
\left[
\begin{array}{c}
 c_3 \\
 c_2
\end{array}
\right],
\qquad
& G''_n(c_1)
\left[
\begin{array}{c}
 c_3 \\
 c_2
\end{array}
\right]
&=
\left[
\begin{array}{c}
 0 \\
 0
\end{array}
\right],
\end{alignedat}
\end{equation}
where
\begin{align*}
G'_n(x)
&\coloneq
\left[
\begin{array}{cc}
 0 & 1 \\
 1 & 0
\end{array}
\right]
 G_n(x)
=\left[
\begin{array}{cc}
 U_{2n+1}\left(x/2\right) & 
 -U_{2n}\left(x/2\right) \\
 U_{2n}\left(x/2\right) & 
 -U_{2n-1}\left(x/2\right)
\end{array}
\right]
\in\SL_2(\F_p),\\
 G''_n(x)
&\coloneq
G'_n(x)-G'_n(x)^{-1}
=U_{2n}\left(\frac{x}{2}\right)
\left[
\begin{array}{cc}
 x & -2 \\
 2 & -x
\end{array}
\right].
\end{align*}

We first find solutions to \eqref{for:explicit minors equation}
in the case where $G''_n(\xi)$ is non-singular for some 
$\xi\in\{a_2,b_3,c_1\}$.
For instance, assume that $G''_n(a_2)$ is non-singular
(the other cases are completely similar).
Then, from \eqref{for:solutions by Chebyshev}, we have 
$a_1=a_3=c_3=c_1=0$, which implies 
$a_2^2=c_2^2=\kappa$ by \eqref{def:Markoff-type equation},
and $b_2=-U_{2n}(0)c_2=(-1)^{n+1}c_2$.
If $\kappa=0$ in $\mathbb{F}_p$, then we have $a_2=0$ and hence 
$X_1=X_2=X_3=Y_1=Y_2=Y_3=O$ with $O=(0,0,0)$.
If $(\frac{\kappa}{p})=-1$,
then there are no solutions to \eqref{for:solutions by Chebyshev}.
Finally, if $(\frac{\kappa}{p})=1$,
then we have $a_2=c_2=\pm\sqrt{\kappa}$ 
and hence $X_1=(0,\pm\sqrt{\kappa},0)=Y_1$.
On the other hand, since 
$X_2=(0,\pm(-1)^{n+1}\sqrt{\kappa},0)$,
we have 
$Y_1=(R_2R_1)^n(X_2)=(0,\mp\sqrt{\kappa},0)$, 
which contradicts the earlier expression of $Y_1$.
Therefore, 
\eqref{for:solutions by Chebyshev} has a solution only when $\kappa=0$ in $\mathbb{F}_p$, which leads to  
$K=\left(\begin{array}{ccc}
O\!\! & O\!\! & O \\
O\!\! & O\!\! & O
\end{array}\right)$.

We next consider the case where 
$G''_n(\xi)$ is singular,
equivalently, 
\begin{equation}
\label{for:necessary condition 1}
 \xi\in Z_n\coloneq
 \left\{x\in\F_p\,\left|\,U_{2n}\left(\frac{x}{2}\right)=0\right.\right\} \quad \text{or} \quad \xi=\pm 2,
\end{equation}
for all $\xi\in\{a_2,b_3,c_1\}$.
This in particular implies that $a_2,b_3,c_1\ne 0$ 
by Lemma~\ref{lem:T_m and U_m properties} (2).
Notice from 
Lemma~\ref{lem:T_m and U_m properties} (3) 
that $\pm 2\in Z_n$ if and only if $2n+1\equiv 0 \pmod{p}$.
Now, consider the equation
\begin{align}
\label{for:desired equation}
\left[
\begin{array}{c}
u \\
v
\end{array}
\right]
=
G'_n(\xi)\left[
\begin{array}{c}
 s \\
 t
\end{array}
\right]
=G'_n(\xi)^{-1}\left[
\begin{array}{c}
 s \\
 t
\end{array}
\right],
\end{align}
which is derived from \eqref{for:solutions by Chebyshev}.



\begin{itemize}
\item 
When $\xi\in Z_n$, it follows from Lemma~\ref{lem:U by A} that 
either $A_{n}(\xi/2)=0$ or $A_{n}(-\xi/2)=0$ hold.
Define $\varepsilon_{\xi}\in\{\pm 1\}$ 
by $A_{n}(\varepsilon_{\xi}\xi/2)=0$. 
Then, Lemma~\ref{lem:U by A} again implies that  
\begin{equation}
\label{for:U_2n+1_2n_2n-1} 
\left(U_{2n+1}\left(\frac{\xi}{2}\right),
U_{2n}\left(\frac{\xi}{2}\right),
U_{2n-1}\left(\frac{\xi}{2}\right)\right)
=(\varepsilon_{\xi},0,-\varepsilon_{\xi}).    
\end{equation}
This leads to $G'_n(\xi)=G'_n(\xi)^{-1}=\varepsilon_{\xi} I_2$,
where $I_2$ is the $2$-by-$2$ identity matrix.
Therefore \eqref{for:desired equation} yields 
$\left[
\begin{array}{c}
u \\
v
\end{array}
\right]
=\varepsilon_{\xi}
\left[
\begin{array}{c}
s \\
t
\end{array}
\right]$.
\item
When $\xi=\pm 2$ with 
$2n+1\not\equiv 0 \pmod{p}$, 
let $\sgn(\xi)\coloneq\xi/2\in\{\pm 1\}$.
By Lemma~\ref{lem:T_m and U_m properties} (3), we have 
\[
 G'_n(\xi)
=\left[
\begin{array}{cc}
 \sgn(\xi) (2n+2) & -(2n+1) \\
 2n+1 & -\sgn(\xi) 2n
\end{array}
\right].
\]
Thus, \eqref{for:desired equation} implies that  
$\left[
\begin{array}{c}
u \\
v
\end{array}
\right]
=t
\left[
\begin{array}{c}
1 \\
\sgn(\xi)
\end{array}
\right]$
.
\end{itemize}
Based on these observations, 
the solutions to \eqref{for:solutions by Chebyshev} 
are obtained as follows.

\begin{itemize}
\item
When $a_2,b_3,c_1\in Z_n$, 
from \eqref{for:solutions by Chebyshev}, we have  
\begin{align*}
\begin{split}
\left[
\begin{array}{c}
 c_3 \\
 c_1
\end{array}
\right]
=\varepsilon_{a_2}
\left[
\begin{array}{c}
 a_1 \\
 b_3
\end{array}
\right],
\quad
\left[
\begin{array}{c}
 a_1 \\
 a_2
\end{array}
\right]
=\varepsilon_{b_3}
\left[
\begin{array}{c}
 b_2 \\
 c_1
\end{array}
\right],
\quad 
\left[
\begin{array}{c}
 b_2 \\
 b_3
\end{array}
\right]
=\varepsilon_{c_1}
\left[
\begin{array}{c}
 c_3 \\
 a_2
\end{array}
\right].
\end{split}
\end{align*}
These imply that
\begin{equation}
\label{for:ab}
a_1
=\varepsilon_{b_3} b_2
=\varepsilon_{b_3} \varepsilon_{c_1} c_3
=\varepsilon_{b_3} \varepsilon_{c_1} \varepsilon_{a_2} a_1,
\quad 
a_2
=\varepsilon_{b_3} c_1
=\varepsilon_{b_3} \varepsilon_{a_2} b_3
=\varepsilon_{b_3} \varepsilon_{a_2} \varepsilon_{c_1} a_2,
\end{equation}
and thus $\varepsilon_{a_2} \varepsilon_{b_3} \varepsilon_{c_1}=1$ 
since $a_2\ne 0$.
By setting $\beta=\varepsilon_{a_2} a_2$, 
which means that $A_n(\beta/2)=0$, 
we can write $(a_2,b_3,c_1)
=(\varepsilon_{a_2} \beta,
\varepsilon_{b_3} \beta,
\varepsilon_{c_1} \beta)$.
Moreover, if we put $\alpha=\varepsilon_{c_1}a_1$,
then we have $f_{\kappa}(\beta,\alpha)=0$
and from \eqref{for:ab}
$(a_1,b_2,c_3)
=(\varepsilon_{c_1} \alpha,
\varepsilon_{a_2} \alpha,
\varepsilon_{b_3} \alpha)$.
Therefore, the matrix $K$ is expressed as 
\begin{align*}
K=
\left(\begin{array}{ccc}
(\varepsilon_{c_1} \alpha,\varepsilon_{a_2} \beta, \varepsilon_{b_3} \beta),\!\! &
(\varepsilon_{c_1} \beta,\varepsilon_{a_2} \alpha, \varepsilon_{b_3} \beta),\!\! & 
(\varepsilon_{c_1} \beta,\varepsilon_{a_2} \beta, \varepsilon_{b_3} \alpha) \\
(\varepsilon_{c_1} \overline{\alpha},\varepsilon_{a_2} \beta, \varepsilon_{b_3} \beta),\!\! & (\varepsilon_{c_1} \beta,\varepsilon_{a_2} \overline{\alpha}, \varepsilon_{b_3} \beta),\!\! & (\varepsilon_{c_1} \beta,\varepsilon_{a_2} \beta, \varepsilon_{b_3} \overline{\alpha})
\end{array}\right).
\end{align*}
This is clearly obtained from $K_{(\alpha,\beta)}$ 
by the double sign change 
$(x,y,z) \mapsto (\varepsilon_{c_1}x,\varepsilon_{a_2}y,\varepsilon_{b_3}z)$ 
\item
When $a_2,b_3\in Z_n$ and $c_1=\pm 2$ with 
$2n+1\not\equiv 0 \pmod{p}$,
it follows from \eqref{for:solutions by Chebyshev} that   
\begin{align*}
\begin{split}
\left[
\begin{array}{c}
 c_3 \\
 c_1
\end{array}
\right]
&=\varepsilon_{a_2}
\left[
\begin{array}{c}
 a_1 \\b_3
\end{array}
\right],
\quad 
\left[
\begin{array}{c}
 a_1 \\
 a_2
\end{array}
\right]
=\varepsilon_{b_3}
\left[
\begin{array}{c}
 b_2 \\
 c_1
\end{array}
\right],
\quad 
\left[
\begin{array}{c}
 b_2 \\
 b_3
\end{array}
\right]
=a_2
\left[
\begin{array}{c}
 1 \\
 \sgn(c_1)
\end{array}
\right].
\end{split}
\end{align*}
These imply 
$c_1
=\varepsilon_{a_2} b_3
=\varepsilon_{a_2} \sgn(c_1) a_2$, 
which yields $a_2/2=\varepsilon_{a_2}$.
Therefore, 
by Lemma~\ref{lem:T_m and U_m properties} (3), 
we have 
$U_{2n-1}(a_2/2)=U_{2n-1}(\varepsilon_{a_2})=\varepsilon_{a_2} 2n$. 
However, this contradicts the fact that
$U_{2n-1}(a_2/2)=-\varepsilon_{a_2}$ from \eqref{for:U_2n+1_2n_2n-1}, given that $2n+1\not\equiv 0 \pmod{p}$.
The cases $b_3,c_1\in Z_n$ and $a_2=\pm 2$, and 
$c_1,a_2\in Z_n$ and $b_3=\pm 2$ are similar.

\item
When $a_2\in Z_n$ and $b_3,c_1=\pm 2$ with 
$2n+1\not\equiv 0 \pmod{p}$, 
it follows from \eqref{for:solutions by Chebyshev} that   
\begin{align*}
\begin{split}
\left[
\begin{array}{c}
 c_3 \\
 c_1
\end{array}
\right]
=\varepsilon_{a_2}
\left[
\begin{array}{c}
 a_1 \\
 b_3
\end{array}
\right],\
\quad 
\left[
\begin{array}{c}
 a_1 \\
 a_2
\end{array}
\right]
=c_1
\left[
\begin{array}{c}
 1 \\
 \sgn(b_3)
\end{array}
\right],
\quad 
\left[
\begin{array}{c}
 b_2 \\
 b_3
\end{array}
\right]
=a_2
\left[
\begin{array}{c}
 1 \\
 \sgn(c_1)
\end{array}
\right].
\end{split}    
\end{align*}
These again imply 
$c_1=\varepsilon_{a_2} b_3
=\varepsilon_{a_2} \sgn(c_1)a_2$,
which leads to the same contradiction as in the previous case.
The cases $b_3\in Z_n$ and $c_1,a_2=\pm 2$, 
and $c_1\in Z_n$ and $a_2,b_3=\pm 2$ are similar.

\item
When $a_2,b_3,c_1=\pm 2$ with 
$2n+1\not\equiv 0 \pmod{p}$, 
it follows from \eqref{for:solutions by Chebyshev} that   
\begin{align*}
\begin{split}
\left[
\begin{array}{c}
 c_3 \\
 c_1
\end{array}
\right]
=b_3
\left[
\begin{array}{c}
 1 \\
 \sgn(a_2)
\end{array}
\right],
\quad 
\left[
\begin{array}{c}
 a_1 \\
 a_2
\end{array}
\right]
=c_1
\left[
\begin{array}{c}
 1 \\
 \sgn(b_3)
\end{array}
\right],
\quad 
\left[
\begin{array}{c}
 b_2 \\
 b_3
\end{array}
\right]
=a_2
\left[
\begin{array}{c}
 1 \\
 \sgn(c_1)
\end{array}
\right].
\end{split}    
\end{align*}
These imply $(a_1,b_2,c_3)=(c_1,a_2,b_3)$ and $\sgn(a_2)\sgn(b_3)\sgn(c_1)=1$.
Since $a_1^2+a_2^2+a_3^2-a_1a_2a_3-\kappa=c_1^2+a_2^2+b_3^2-c_1a_2b_3-\kappa=4-\kappa$,
there are no solutions to \eqref{for:explicit minors equation} 
unless $\kappa=4$. 
When $\kappa=4$, 
we find that $K$ is a double sign change of
$
\left(\begin{array}{ccc}
P\!\! & P\!\! & P \\
P\!\! & P\!\! & P
\end{array}\right)
$, where $P=(2,2,2)$.
\end{itemize}
This completes the proof of 
$\Knkpall\subset\widetilde{\mathcal{K}}^{\text{all}}_{n,\kappa}(p)$.
For the reverse inclusion
$\Knkpall\supset\widetilde{\mathcal{K}}^{\text{all}}_{n,\kappa}(p)$,
it suffices to verify that
$K_{(\alpha,\beta)}$ with $(\alpha,\beta)\in \Tnkpall$ 
satisfies \eqref{for:solutions by Chebyshev}, that is, 
\[
\left[
\begin{array}{c}
 \alpha \\
 \beta
\end{array}
\right]
=G'_n(\beta)
\left[
\begin{array}{c}
 \alpha \\
 \beta
\end{array}
\right],
\qquad 
G''_n(\beta)
\left[
\begin{array}{c}
 \alpha \\
 \beta
\end{array}
\right]
=
\left[
\begin{array}{c}
 0 \\
 0
\end{array}
\right].
\]
These relations indeed hold because 
$G'_n(\beta)=I_2$ and $G'_n(\beta)=O_2$,
where $O_2$ is the $2$-by-$2$ zero matrix.

(2) To study $\Knkpdist$, 
it suffices to determine when any coincidence occurs among the triples in
$K_{(\alpha,\beta)}$ for $(\alpha,\beta)\in \Tnkpall$. 
To see this, we first note that, under the condition $f_{\kappa}(\beta,\alpha)=0$, 
we have $(\beta-\alpha)(\beta-\overline{\alpha})=0$ if and only if $\beta^2(3-\beta)=\kappa$, and 
$\alpha=\overline{\alpha}$ if and only if
$\beta^2(8-\beta^2)=4\kappa$.
Since these conditions are equivalent to $X_i=X_j$ or $Y_i=Y_j$ for some $1\le i\ne j\le 3$, and $X_i=Y_i$ for some $1\le i\le 3$, respectively, 
it follows that $K_{(\alpha,\beta)}\in \Knkpdist$ if and only if $\beta^2(3-\beta)\ne \kappa$ and $\beta^2(8-\beta^2)\ne 4\kappa$, 
which is equivalent to $(\alpha,\beta)\in\Tnkpdist$.
Finally, we observe that $X_i=Y_j$ for some $1\le i\ne j\le 3$
implies $\beta=\alpha=\overline{\alpha}$,
which occurs only when $\kappa=4$.
In this specific case, 
we have $\beta=\alpha=\overline{\alpha}=2$,
which implies $K_{(\alpha,\beta)}\notin\Knkpdist$.
This completes the proof.
\end{proof}

\begin{remark}
\label{rem:Cayley-cubic}
When $\kappa=4$, 
for which \eqref{def:Markoff-type equation} defines a degenerate surface, 
we observe that $\mathcal{K}^{\text{prop}}_{n,4}(p)=\emptyset$ even though $\mathcal{K}^{\text{all}}_{n,4}(p)\ne\emptyset$.
Actually, take any 
$K_{(\alpha,\beta)}\in\mathcal{K}^{\text{all}}_{n,4}(p)$
for some $(\alpha,\beta)\in T^{\text{all}}_{n,4}(p)$.
Since $f_4(\beta,\alpha)=(\alpha-2)(\overline{\alpha}-2)$,
it follows that $\alpha=2$ or $\overline{\alpha}=2$.
If $\alpha=2$, then $X_1=(2,\beta,\beta)$,
which implies $R_2(X_1)=R_3(X_1)=X_1$. 
This means that $Y_1=R_1R_2(X_1)=R_1R_3(X_1)$
lies on both the $X_1$-$Y_2$-path and the $X_1$-$Y_3$-path, 
and therefore $K_{(\alpha,\beta)}$ is not proper.
The case $\overline{\alpha}=2$ is similar.  
For this reason, throughout this paper, we assume $\kappa \neq 4$. 
\end{remark}


\subsection{Rewriting the conditions for $\Tnkpall$}



To study the system of algebraic equations in $\Tnkpall$,
let 
\[
\Inkp
\coloneq\left\langle A_n\left(\frac{x}{2}\right),f_{\kappa}(x,y)
\right\rangle
\]
denote the ideal of $\Fp[x,y]$ generated by 
$A_n(x/2)$ and $f_{\kappa}(x,y)$.
To study $\Tnkpall$,
we need a good basis of $\Inkp$.
To do that, we recall the definition of the resultant of two polynomials:
Let $R$ be a commutative ring.
For polynomials $f(x)=\sum^{n}_{i=0}a_ix^i$ 
and $g(x)=\sum^{m}_{i=0}b_ix^i$ in $R[x]$,
the resultant $\Res(f,g)\in R$ of $f$ and $g$
is defined by  
\[
\Res(f,g)
\coloneq
\left|
\begin{array}{ccccccc}
a_n & a_{n-1} & \cdots & a_0 & & &  \\
 & a_n & a_{n-1} & \cdots & a_0 & &  \\
 & & \ddots & & & \ddots &  \\
 & & & a_n & a_{n-1} & \cdots & a_0  \\
 \hline 
b_m & b_{m-1} & \cdots & b_0 & & &  \\
 & b_m & b_{m-1} & \cdots & b_0 & &  \\
 & & \ddots & & & \ddots &  \\
 & & & b_m & b_{m-1} & \cdots & b_0
\end{array}
\right|.
\]
Now, define  
\begin{align*}
B_{n,\kappa}(y)
&\coloneq\Res_x\left(A_n\left(\frac{x}{2}\right),f_{\kappa}(x,y)\right)
=\left|
\begin{array}{ccccccc}
a_{n,0} & a_{n,1} & \cdots & a_{n,n-1} & a_{n,n}
\\
& a_{n,0} & a_{n,1} & \cdots & a_{n,n-1} & a_{n,n}
\\[5pt]
\hline 
 2-y & 0 & y^2-\kappa &  \\
 & 2-y & 0 & y^2-\kappa & \\
 & & \ddots & \ddots & \ddots \\
 & & & 2-y & 0 & y^2-\kappa & \\
\end{array}
\right|,
\end{align*} 
where $\Res_x$ denotes the resultant with respect to the variable $x$, 
and the coefficients $a_{n,j}$ are defined as in \eqref{for:explicit expression of Am}.
Note that $B_{n,\kappa}(y)$ 
is a monic polynomial in $y$ of degree $2n$.
The following proposition enables us to obtain $(\alpha,\beta)\in\Tnkpall$ by solving the equation $B_{n,\kappa}(y)=0$.

\begin{prop} 
Let $\kappa\in\mathbb{Z}\setminus\{4\}$ 
and $n\in\mathbb{Z}_{>0}$ satisfying 
$2n+1\not\equiv 0 \pmod{p}$.
Define 
\begin{equation}
\eta_{n,\kappa}
\coloneq \Res\left(A_n\left(\frac{x}{2}\right),
x^4-8x^2+4\kappa\right)\in\Fp.   
\end{equation}
If $\eta_{n,\kappa}\ne 0$,
then there exists a unique polynomial $C_{n,\kappa}(y)\in\Fp[y]$ of degree $<2n$ such that 
\begin{equation}
\label{for:grobner basis shape type}
\Inkp=\left\langle B_{n,\kappa}(y),x-\frac{1}{(\kappa-4)^{n-1}}C_{n,\kappa}(y)\right\rangle.
\end{equation}
\end{prop}
\begin{proof}
It is clear that $\Inkp$ 
is a zero-dimensional ideal.
Since $D(A_n(x/2))=(2n+1)^{n-1}$
by Lemma~\ref{lem:Am properties} (4),
our assumption on $n$ ensures that
$A_n(x/2)$ has no multiple roots.
Moreover, since $D_y(f_{\kappa}(x,y))=x^4-8x^2+4\kappa$, where $D_y$ denotes the discriminant with respect to the variable $y$,
the assumption $\eta_{n,\kappa}\ne 0$
implies that $f_{\kappa}(\gamma,y)$ has no multiple roots 
for any root $\gamma$ of $A_n(x/2)=0$.
These observations show that $\Inkp$ is a radical ideal of $\Fp[x,y]$.
Therefore, we can apply the so-called shape Lemma (see, e.g.~\cite[Theorem~3.7.25)]{KR2000}),
which guarantees the existence of 
unique polynomials $g(y),h(y)\in\Fp[y]$ such that 
$g(y)$ is monic,
$\deg h(y)<\deg g(y)$, 
and $I_{n,\kappa}(p)=\langle g(y),x-h(y)\rangle$
(such a basis is called a shape basis, which is a special type of reduced Gr\"obner basis).
Since $g(y)$ generates the elimination ideal $\Inkp\cap \Fp[y]$, we may take it to be the resultant $B_{n,\kappa}(y)$.
Finally, we put $C_{n,\kappa}(y)\coloneq(\kappa-4)^{n-1}h(y)$,
where the factor is introduced as a normalization.
\end{proof}

Now, by \eqref{for:grobner basis shape type},
we have $A_n(\beta/2)=f_{\kappa}(\beta,\alpha)=0$ if and only if $B_{n,\kappa}(\alpha)=\beta-C_{n,\kappa}(\alpha)/(\kappa-4)^{n-1}=0$.
Moreover, if $\xi_{n,\kappa}\ne 0$ and $\eta_{n,\kappa}\ne 0$, where  
\begin{equation*}
\xi_{n,\kappa}
\coloneq\Res\left(A_n\left(\frac{x}{2}\right),x^3-3x^2+\kappa\right)
\in\Fp,
\end{equation*}
then neither of the systems
$A_{n}(x/2)=x^2(3-x)-\kappa=0$ and 
$A_{n}(x/2)=x^2(8-x^2)-4\kappa=0$ 
has a common root. 
Therefore, one obtains the following:

\begin{cor}
\label{cor:all equal dist}
Let $\kappa\in\mathbb{Z}\setminus\{4\}$ and 
$n\in\mathbb{Z}_{>0}$ satisfying $2n+1\not\equiv 0 \pmod{p}$.
Assume that $\xi_{n,\kappa}\eta_{n,\kappa}\ne 0$.
\begin{itemize}
\item[$(1)$]
We have   
\begin{equation}
\label{for:Tnkdp by B and C}
\Tnkpall=\Tnkpdist=\left\{\left(\alpha,\beta\right)\in \Fp\times \Fpm\,\left|\,
B_{n,\kappa}(\alpha)=0,\ 
\beta=\frac{C_{n,\kappa}(\alpha)}{(\kappa-4)^{n-1}}\right.
\right\}.    
\end{equation}
\item[$(2)$]
If $(\alpha,\beta)\in \Tnkpdist$, then 
$B_{n,\kappa}(\alpha)=B_{n,\kappa}(\overline{\alpha})=0$.
Hence, the roots of $B_{n,\kappa}(y)=0$ appear in pairs.
\end{itemize}
\end{cor}
\begin{proof}
The above observations show all assertions except for
$B_{n,\kappa}(\overline{\alpha})=0$. 
The remaining one follows from 
$B_{n,\kappa}(y)\in \left\langle A_n\left(x/2\right),f_{\kappa}(x,y)
\right\rangle$ together with the identity
$f_{\kappa}(x,y)=f_{\kappa}(x,x^2-y)$.
\end{proof}

\begin{example}
\label{ex-Grobner}
A direct computation 
shows that
\begin{align*}
\xi_{1,\kappa}
&=\kappa-4,\\
\eta_{1,\kappa}
&=4 \kappa-7,\\
 B_{1,\kappa}(y)
&=y^2-y-\kappa+2,\\
 C_{1,\kappa}(y)
&=-1, \\[5pt]
\xi_{2,\kappa}
&=\kappa^2-13 \kappa+11,\\
\eta_{2,\kappa}
&=16 \kappa^2-68 \kappa+41,\\
 B_{2,\kappa}(y)
&=y^4-3y^3-(2 \kappa-7) y^2+(3 \kappa-4) y+\kappa^2-6 \kappa+4, \\
 C_{2,\kappa}(y)
&=-y^3+y^2+(\kappa-1) y-2, \\[5pt]
\xi_{3,\kappa}
&=\kappa^3-19 \kappa^2+55 \kappa-29,\\
\eta_{3,\kappa}
&=64 \kappa^3-432 \kappa^2+776 \kappa-239,\\
 B_{3,\kappa}(y)
&=y^6-5 y^5-(3\kappa-16) y^4+5(2 \kappa-5) y^3
+\left(3 \kappa^2-26 \kappa+30\right) y^2
\\
&\quad -\left(5 \kappa^2-24 \kappa+12\right) y-\kappa^3+10 \kappa^2-24 \kappa+8,\\
 C_{3,\kappa}(y)
&=-2 y^5-(\kappa-10) y^4+(7 \kappa-24) y^3
+\left(2 \kappa^2-20 \kappa+26\right) y^2\\
&\quad -\left(5 \kappa^2-25 \kappa+12\right) y
-\kappa^3+9 \kappa^2-18 \kappa, \\[5pt]
\xi_{4,\kappa}
&=\kappa^4-25 \kappa^3+150 \kappa^2-283 \kappa+76,\\
\eta_{4,\kappa}
&=256 \kappa^4-2368 \kappa^3+7152 \kappa^2-7432 \kappa+1393,\\
 B_{4,\kappa}(y)
&=y^8-7 y^7-(4\kappa-29) y^6+7(3 \kappa-10) y^5
+\left(6 \kappa^2-72 \kappa+121\right) y^4
-\left(21 \kappa^2-130 \kappa+128\right) y^3\\
&\quad -\left(4 \kappa^3-57 \kappa^2+180 \kappa-104\right) y^2
+\left(7 \kappa^3-60 \kappa^2+120 \kappa-32\right) y
+\kappa^4-14 \kappa^3+60 \kappa^2-80 \kappa+16,\\
 C_{4,\kappa}(y)
&=-(\kappa+1) y^7+(3 \kappa+13) y^6+\left(3 \kappa^2-\kappa-59\right) y^5
-\left(4 \kappa^2+47 \kappa-152\right) y^4\\
&\quad -\left(3 \kappa^3+\kappa^2-145 \kappa+217\right) y^3
-\left(\kappa^3-63 \kappa^2+274 \kappa-182\right) y^2\\
&\quad +\left(\kappa^4+3 \kappa^3-88 \kappa^2+220 \kappa-60\right) y
+2 \kappa^4-28 \kappa^3+116 \kappa^2-136 \kappa+8.
\end{align*}
\end{example}

\begin{remark}
\label{for:dibisibility}
One sees that 
$B_{n,\kappa}(x)$ is divided by $B_{s,\kappa}(x)$ 
if $n=(2s+1)l+s$ for some integer $l\ge 0$,
which implies that 
$T^{\text{all}}_{n,\kappa}(p)\ne\emptyset$ 
whenever $T^{\text{all}}_{s,\kappa}(p)\ne\emptyset$.
For example, 
when $n=4$, $s=1$ and $l=1$, we have 
\begin{align*}
\frac{B_{4,\kappa}(y)}{B_{1,\kappa}(y)}
&=
-y^6
+6y^5
+3(\kappa-7)y^4
-(12\kappa-37)y^3\\
&\qquad 
-3(\kappa^2-11 \kappa+14)y^2
+6(\kappa^2-6 \kappa+2)y
+\kappa^3-12\kappa^2+36\kappa-8.  
\end{align*}
This divisibility can be proved as follows.
From \cite[Theorem~2.3]{LW2011}, 
$A_n(x)$ is divisible by $A_s(x)$ in $\mathbb{Z}[x]$ if and only if $n=(2s+1)l+s$ for some integer $l\ge 0$.
Thus, by writing $A_n(x)=A_s(x)Q_{n,s}(x)$ for some $Q_{n,s}(x)\in\mathbb{Z}[x]$ and applying the multiplicative property of the resultant $\Res(f(x)g(x),h(x))=\Res(f(x),h(x))\Res(g(x),h(x))$, 
we obtain $B_{n,\kappa}(y)=B_{s,\kappa}(y)
\Res_x\left(Q_{n,s}\left(x/2\right),f_{\kappa}(x,y)\right)$.
For the same reason, 
$\xi_{n,\kappa}$ and $\eta_{n,\kappa}$ are divisible by $\xi_{s,\kappa}$ and $\eta_{s,\kappa}$, respectively, for such an integer $s$.
\end{remark}

\subsection{Properness of $K_{(\alpha,\beta)}$}
\label{sec:K33 minors}

Now, for $\kappa\in\mathbb{Z}\setminus\{4\}$,
our task is to determine 
when $\Knkpprop\ne\emptyset$ for a given prime $p>3$ and $n\in\mathbb{Z}_{>0}$.
To this end, we investigate when a configuration 
$K=\left(\begin{array}{ccc}
X_1,\!\! & X_2,\!\! & X_3 \\
Y_1,\!\! & Y_2,\!\! & Y_3
\end{array}\right)\in\Knkpdist$ 
is proper in our sense.
We may assume that $K$ is of the form of $K=K_{(\alpha,\beta)}$ for some $(\alpha,\beta)\in \Tnkpdist$. 
It is clear that neither $X_k$ nor $Y_k$ is an internal vertex of the $X_i$-$Y_j$-path for any $\{i,j,k\}=\{1,2,3\}$. 
Moreover, one easily observes that $X_i$ (resp. $X_j$) is an internal vertex of the $X_i$-$Y_j$-path if and only if $Y_i$ (resp. $Y_j$) is. Therefore, by the symmetry of $\Gkp$, we see that $K$ is proper if and only if neither $X_1$ nor $X_2$ is an internal vertex of the $X_1$-$Y_2$ path.
Now, assume that $n\in\mathbb{Z}_{>0}$ 
satisfies $2n+1\not\equiv 0 \pmod{p}$,
which implies that $\alpha\ne 2$ and $\beta\ne \pm 2$ for any 
$(\alpha,\beta)\in\Tnkpdist$. 
Under these conditions,
we consider the following four cases.

\begin{itemize}
\item[(i-a)] 
The case where there exists $0<m<n$ such that 
$(R_1R_2)^m(X_1)=X_1$. 
From Lemma~\ref{lem:F G}, we have 
$F_m(\beta)
\left[
\begin{array}{c}
 \alpha\\ 
 \beta
\end{array}
\right]
=\left[
\begin{array}{c}
 \alpha \\ 
 \beta
\end{array}
\right]$, 
which is equivalent to 
\begin{equation*}
\left\{
\begin{array}{rcl}
\left(U_{2m}\left(\frac{\beta}{2}\right)-1\right)\alpha
&\!\!\!\!=\!\!\!\!&U_{2m-1}\left(\frac{\beta}{2}\right)\beta, \\[5pt]
U_{2m-1}\left(\frac{\beta}{2}\right)\alpha
&\!\!\!\!=\!\!\!\!&\left(U_{2m-2}\left(\frac{\beta}{2}\right)+1\right)\beta. 
\end{array}
\right.
\end{equation*}
Moreover, by Lemma~\ref{lem:T_m and U_m properties} (4), (5) and (6),
these are equivalent to
\[
\left\{
\begin{array}{rcl}
U_{m-1}\left(\frac{\beta}{2}\right)\left(T_{m+1}\left(\frac{\beta}{2}\right)\alpha-T_{m}\left(\frac{\beta}{2}\right)\beta\right)
&\!\!\!\!=\!\!\!\!& 0,\\
U_{m-1}\left(\frac{\beta}{2}\right)\left(T_{m}\left(\frac{\beta}{2}\right)\alpha-T_{m-1}\left(\frac{\beta}{2}\right)\beta\right)
&\!\!\!\!=\!\!\!\!& 0.
\end{array}
\right.
\]
This implies $U_{m-1}(\beta/2)=0$, because, otherwise, 
we have $T_{m+1}(\beta/2)\alpha-T_{m}(\beta/2)\beta=T_{m}(\beta/2)\alpha-T_{m-1}(\beta/2)\beta=0$. 
Using \eqref{for:recursion for T_m} repeatedly, 
we obtain $T_{1}(\beta/2)\alpha-T_0(\beta/2)\beta=(\alpha-2)\beta/2=0$,
which contradicts the fact that $\alpha\ne 2$ and $\beta\ne 0$. 
Hence, 
since $\beta$ is a common root of $A_n(x/2)$ and $U_{m-1}(x/2)$,
we have the following.

\begin{obs}
If there exists $0<m<n$ such that $(R_1R_2)^m(X_1)=X_1$,
then we have 
\[
\Res_x\left(A_n\left(\frac{x}{2}\right),U_{m-1}\left(\frac{x}{2}\right)\right)=0.
\]
\end{obs}

\item[(i-b)]
The case where there exists $0\le m<n$ such that 
$(R_2(R_1R_2)^m)(X_1)=X_1$.
From Lemma~\ref{lem:F G}, we have 
$G_m(\beta)
\left[
\begin{array}{c}
 \alpha\\ 
 \beta
\end{array}
\right]
=\left[
\begin{array}{c}
 \alpha \\ 
 \beta
\end{array}
\right]$, 
which is equivalent to 
\begin{equation*}
\left\{
\begin{array}{rcl}
\left(U_{2m}\left(\frac{\beta}{2}\right)-1\right)\alpha
&\!\!\!\!=\!\!\!\!&U_{2m-1}\left(\frac{\beta}{2}\right)\beta, \\[5pt]
U_{2m+1}\left(\frac{\beta}{2}\right)\alpha
&\!\!\!\!=\!\!\!\!&\left(U_{2m}\left(\frac{\beta}{2}\right)+1\right)\beta.
\end{array}
\right.
\end{equation*}
Moreover, by Lemma~\ref{lem:T_m and U_m properties} (4), (5) and (6),
these are equivalent to
\[
\left\{
\begin{array}{rcl}
U_{m-1}\left(\frac{\beta}{2}\right)\left(T_{m+1}\left(\frac{\beta}{2}\right)\alpha-T_{m}\left(\frac{\beta}{2}\right)\beta\right)
&\!\!\!\!=\!\!\!\!& 0,\\
U_{m}\left(\frac{\beta}{2}\right)\left(T_{m+1}\left(\frac{\beta}{2}\right)\alpha-T_{m}\left(\frac{\beta}{2}\right)\beta\right)
&\!\!\!\!=\!\!\!\!& 0.
\end{array}
\right.
\]
This implies
\begin{equation}
\label{eq:i-b}
T_{m+1}\left(\frac{\beta}{2}\right)\alpha
-T_{m}\left(\frac{\beta}{2}\right)\beta=0,
\end{equation}
because, otherwise, 
we have $U_{m-1}(\beta/2)=U_{m}(\beta/2)=0$. 
Using \eqref{for:recursion for U_m} repeatedly, 
we obtain $U_{0}(\beta/2)=0$,
which contradicts the fact that $U_{0}(\beta/2)=1$. 
Now, substituting $\beta^2=(\alpha^2-\kappa)/(\alpha-2)$ into \eqref{eq:i-b} and noticing $\beta\ne 0$, we see that $f_{m,\kappa}(\alpha)=0$,
where $h_{m,\kappa}(y)\in\Fp[y]$ is a monic polynomial of degree $2\lfloor(m+1)/2\rfloor+1$ defined by 
\[
 h_{m,\kappa}(y)
\coloneq 2\left(\sum^{\lfloor\frac{m+1}{2}\rfloor}_{j=0}t_{m+1,j}(y^2-\kappa)^{\lfloor\frac{m+1}{2}\rfloor-j}(y-2)^jy
-
\sum^{\lfloor\frac{m}{2}\rfloor}_{j=0}t_{m,j}(y^2-\kappa)^{\lfloor\frac{m+1}{2}\rfloor-j}(y-2)^j\right).
\]
Hence, 
since $\alpha$ is a common root of $B_{n,\kappa}(y)$ and $h_{m,\kappa}(y)$, we have the following.
\begin{obs}
If there exists $0\le m<n$ such that 
$(R_2(R_1R_2)^m)(X_1)=X_1$,
then we have 
\[
\Res_y(B_{n,\kappa}(y),h_{m,\kappa}(y))=0.
\]
\end{obs}


\item[(ii-a)] 
The case where there exists $0<m<n$ such that 
$(R_1R_2)^m(X_1)=X_2$. 
From Lemma~\ref{lem:F G}, we have 
$F_m(\beta)
\left[
\begin{array}{c}
 \alpha\\ 
 \beta
\end{array}
\right]
=\left[
\begin{array}{c}
 \beta \\
 \alpha  
\end{array}
\right]$, 
which is equivalent to 
\begin{equation}
\label{for:R1R2X1 X2}
\left\{
\begin{array}{rcl}
U_{2m}\left(\frac{\beta}{2}\right)\alpha
&\!\!\!\!=\!\!\!\!&\left(U_{2m-1}\left(\frac{\beta}{2}\right)+1\right)\beta, \\[5pt]
\left(U_{2m-1}\left(\frac{\beta}{2}\right)-1\right)\alpha
&\!\!\!\!=\!\!\!\!&U_{2m-2}\left(\frac{\beta}{2}\right)\beta.  
\end{array}
\right.
\end{equation}
Notice that if $U_{2m}(\beta/2)\ne 0$,
then the second equation is a consequence of the first.
Now, by substituting 
$\beta=C_{n,\kappa}(\alpha)/(\kappa-4)^n$
and $\beta^2=(\alpha^2-\kappa)/(\alpha-2)$ into the first equation of \eqref{for:R1R2X1 X2}, 
we see that $g_{n,m,\kappa}(\alpha)=0$,
where $g_{n,m,\kappa}(y)\in\Fp[y]$ is a polynomial of degree $<2n+m$ defined by  
\begin{align*}
 g_{n,m,\kappa}(y)
&\coloneq(y-2)^{m}C_{n,\kappa}(y)\\
&\quad 
-(\kappa-4)^{n-1}\left\{\sum^{m}_{j=0}u_{2m,j}(y^2-\kappa)^{m-j}(y-2)^{j}y-\sum^{m-1}_{j=0}u_{2m-1,j}(y^2-\kappa)^{m-j}(y-2)^{j}\right\}.    
\end{align*}
Hence, 
since $\alpha$ is a common root of 
$B_{n,\kappa}(y)$ and $g_{n,m,\kappa}(y)$,
we have the following.

\begin{obs}
If there exists $0<m<n$ such that 
$(R_1R_2)^m(X_1)=X_2$, then we have
\[
\Res_y(B_{n,\kappa}(y),g_{n,m,\kappa}(y))=0.
\]
\end{obs}

\item[(ii-b)] 
The case where there exists $0\le m<n$ such that
$(R_2(R_1R_2)^m)(X_1)=X_2$.
From Lemma~\ref{lem:F G}, we have 
$G_m(\beta)
\left[
\begin{array}{c}
 \alpha\\ 
 \beta
\end{array}
\right]
=\left[
\begin{array}{c}
 \beta \\ 
 \alpha   
\end{array}
\right]$, 
which is equivalent to 
\begin{equation*}
\left\{
\begin{array}{rcl}
U_{2m}\left(\frac{\beta}{2}\right)\alpha
&\!\!\!\!=\!\!\!\!&\left(U_{2m-1}\left(\frac{\beta}{2}\right)+1\right)\beta, \\[5pt]
\left(U_{2m+1}\left(\frac{\beta}{2}\right)-1\right)\alpha
&\!\!\!\!=\!\!\!\!&U_{2m}\left(\frac{\beta}{2}\right)\beta. 
\end{array}
\right.
\end{equation*}
Moreover, by Lemma~\ref{lem:U by A},
these are equivalent to
\[
\left\{
\begin{array}{rcl}
A_{m}\left(\frac{\beta}{2}\right)\left(A_{m}\left(-\frac{\beta}{2}\right)\alpha+A_{m-1}\left(-\frac{\beta}{2}\right)\beta\right)
&\!\!\!\!=\!\!\!\!& 0,\\
A_{m}\left(\frac{\beta}{2}\right)\left(A_{m+1}\left(-\frac{\beta}{2}\right)\alpha+A_{m}\left(-\frac{\beta}{2}\right)\beta\right)
&\!\!\!\!=\!\!\!\!& 0.
\end{array}
\right.
\]
Similar to the argument in case (i-a), 
this implies $A_{m}(\beta/2)=0$, because, otherwise, 
we have $A_{m}(-\beta/2)\alpha+A_{m-1}(-\beta/2)\beta=A_{m+1}(-\beta/2)\alpha+A_{m}(-\beta/2)\beta=0$. 
Using \eqref{for:recursion formula for Am} repeatedly, 
we obtain 
$A_{2}(-\beta/2)\alpha+A_{1}(-\beta/2)\beta=(\beta^2-\beta-1)\alpha-(\beta-1)\beta=0$
and 
$A_{1}(-\beta/2)\alpha+A_{0}(-\beta/2)\beta=-(\beta-1)\alpha+\beta=0$.
Hence, 
since $\beta$ is a common root of $A_n(x/2)$ and $A_{m}(x/2)$,
we have the following.

\begin{obs}
If there exists $0\le m<n$ such that
$(R_2(R_1R_2)^m)(X_1)=X_2$, then we have 
\[
\Res_x\left(A_n\left(\frac{x}{2}\right),A_{m}\left(\frac{x}{2}\right)\right)=0.
\]
\end{obs}

\end{itemize}

    

Based on these observations, 
we obtain the following result.

\begin{prop}
\label{porp:Kd=Kdd for n=1,2,3}
Let $\kappa\in\mathbb{Z}\setminus\{4\}$
and $n\in\mathbb{Z}_{>0}$ satisfying $2n+1\not\equiv 0 \pmod{p}$.
\begin{itemize}
\item[$(1)$] 
Assume that $\xi_{1,\kappa}=\kappa-4\ne 0$ and 
$\eta_{1,\kappa}=4\kappa-7\ne 0$.
Then  
\[
\mathcal{K}^{\mathrm{all}}_{1,\kappa}(p)
=\mathcal{K}^{\mathrm{dist}}_{1,\kappa}(p)
=\mathcal{K}^{\mathrm{prop}}_{1,\kappa}(p).
\]
\item[$(2)$] 
Assume that 
$\xi_{2,\kappa}=\kappa^2-13 \kappa+11\ne 0$ and 
$\eta_{2,\kappa}=16 \kappa^2-68 \kappa+41\ne 0$.
Then 
\[
\mathcal{K}^{\mathrm{all}}_{2,\kappa}(p)
=\mathcal{K}^{\mathrm{dist}}_{2,\kappa}(p)
=\mathcal{K}^{\mathrm{prop}}_{2,\kappa}(p).
\] 
\item[$(3)$] 
Assume that 
$\xi_{3,\kappa}=\kappa^3-19 \kappa^2+55 \kappa-29\ne 0$, 
$\eta_{3,\kappa}=64 \kappa^3-432 \kappa^2+776 \kappa-239\ne 0$,
and $\kappa^3-26 \kappa^2+111 \kappa-43\ne 0$.
Then, 
\[
\mathcal{K}^{\mathrm{all}}_{3,\kappa}(p)
=\mathcal{K}^{\mathrm{dist}}_{3,\kappa}(p)
=\mathcal{K}^{\mathrm{prop}}_{3,\kappa}(p).
\] 
\end{itemize}
\end{prop}
\begin{proof}
As shown in Corollary~\ref{cor:all equal dist}, $\Knkpall=\Knkpdist$ if $\xi_{n,\kappa}\eta_{n,\kappa}\ne 0$.
Therefore,
based on our previous observations for each $n=1,2,3$, 
it suffices to show that no such $m$ exists in any of the four cases above, which guarantees that $\Knkpdist=\Knkpprop$. 
\begin{itemize}
\item[(1)]
The case $n=1$.
\begin{itemize}
\item[(i-b)] 
Since $\Res_y(B_{1,\kappa}(y),h_{0,\kappa}(y))=-(\kappa-4)$, 
the assumption ensures that no such $m$ exists.
\item[(ii-b)] 
Since 
$\Res_x(A_1(x/2),A_{0}(x/2))=1$, no such $m$ exists.
\end{itemize} 
\item[(2)]
The case $n=2$.    
\begin{itemize}
\item[(i-a)] 
Since 
$\Res_x(A_2(x/2),U_{0}(x/2))=1$, no such $m$ exists.
\item[(i-b)] 
Since $\Res_y(B_{2,\kappa}(y),h_{0,\kappa}(y))=(\kappa-4)^2$ and 
$\Res_y(B_{2,\kappa}(y),h_{1,\kappa}(y))
=(\kappa-4)^2\xi_{2,\kappa}$,
no such $m$ exists by the assumptions.
\item[(ii-a)] 
Since $\Res_y(B_{2,\kappa}(y),g_{2,1,\kappa}(y))=(\kappa-4)^8$,
no such $m$ exists.
\item[(ii-b)] 
Since 
$\Res_x(A_2(x/2),A_{0}(x/2))=1$ and $\Res_x(A_2(x/2),A_{1}(x/2))=-1$,
no such $m$ exists.
\end{itemize}
\item[(3)]
The case $n=3$.      
\begin{itemize}
\item[(i-a)] 
Since 
$\Res_x(A_3(x/2),U_{0}(x/2))=\Res_x(A_3(x/2),U_{1}(x/2))=1$,
no such $m$ exists.  
\item[(i-b)] 
Since 
\begin{align*}
\Res_y(B_{3,\kappa}(y),h_{0,\kappa}(y))
&=-(\kappa-4)^3,\\
\Res_y(B_{3,\kappa}(y),h_{1,\kappa}(y))
&=(\kappa-4)^3 (\kappa^3-26 \kappa^2+111 \kappa-43),\\
\Res_y(B_{3,\kappa}(y),h_{2,\kappa}(y))
&=(\kappa-4)^3 \xi_{3,\kappa},
\end{align*} 
no such $m$ exists by the assumptions.
\item[(ii-a)] 
Since 
\begin{align*}
\Res_y(B_{3,\kappa}(y),g_{3,1,\kappa}(y))
&=(\kappa-4)^{15} (\kappa^3-26 \kappa^2+111 \kappa-43),\\
\Res_y(B_{3,\kappa}(y),g_{3,2,\kappa}(y))
&=-(\kappa-4)^{21}, 
\end{align*}
no such $m$ exists by the assumptions.
\item[(ii-b)] 
Since 
$\Res_x(A_3(x/2),A_{1}(x/2))=1$ and 
$\Res_x(A_3(x/2),A_{1}(x/2))
=\Res_x(A_3(x/2),A_{2}(x/2))=-1$, 
no such $m$ exists.
\end{itemize}
\end{itemize}
\end{proof}

\begin{remark}
The same discussion cannot be extended to the general case $n\ge 4$.
In fact, direct computation shows that the following resultants vanish:
\begin{align*}
\Res_y(B_{4,\kappa}(y),g_{4,1,\kappa}(y))&=0,\\
\Res_y(B_{7,\kappa}(y),g_{7,m,\kappa}(y))&=0
\quad \text{for $m=1,2,4$,}\\
\Res_y(B_{10,\kappa}(y),g_{10,m,\kappa}(y))&=0
\quad \text{for $m=1,3,4,7$.} 
\end{align*}
\end{remark}

\section{Explicit $K_{3,3}$-subdivisions in $\Gkp$}
\label{sect:non-planarity}

In this section, 
for certain specific $n$ and $p$, 
we show that $\Knkpprop \ne \emptyset$, that is, 
there exists an explicit $K_{3,3}$-subdivision in $\Gkp$ 
via Lemma~\ref{lem:properties of Kndd}.
As a consequence, one can also deduce the non-planarity of $\Gkp$ as well as the existence of short cycles within the graph 
from Corollaries~\ref{cor:sufficient condition non-planarity} and \ref{cor:cycles}, respectively. 

\subsection{The case $n=1$}
\label{subsec:n=1}
We first consider the case $n=1$.
Suppose that $\xi_{1,\kappa}=\kappa-4\ne 0$ and $\eta_{1,\kappa}=4\kappa-7\ne 0$.
Since $B_{1,\kappa}(y)=y^2-y-\kappa+2$
and $C_{1,\kappa}(y)=-1$,
we have from Corollary~\ref{cor:all equal dist}
\begin{equation}
\label{eq-solution-n=1}
T^{\text{all}}_{1,\kappa}(p)
=T^{\text{dist}}_{1,\kappa}(p)
=
\begin{cases}
\left\{\left(\alpha_{\kappa},-1\right),\left(\overline{\alpha_{\kappa}},-1\right)\right\} & \left(\dfrac{\eta_{1,\kappa}}{p}\right)=1, \\[7pt]
\emptyset & \text{otherwise},
\end{cases}
\end{equation}
where $\alpha_{\kappa}\coloneq (1+\sqrt{\eta_{1,\kappa}})/2$
and $\overline{\alpha_{\kappa}}\coloneq (1-\sqrt{\eta_{1,\kappa}})/2$.
Therefore, we have   
\[
\mathcal{K}^{\text{all}}_{1,\kappa}(p)
=\mathcal{K}^{\text{dist}}_{1,\kappa}(p)
=\mathcal{K}^{\text{prop}}_{1,\kappa}(p)
=
\begin{cases}
\left\{
\left.
\sigma (K_{(\alpha_{\kappa},-1)}),
\sigma (K_{(\overline{\alpha_{\kappa}},-1)})
\,\right|\,\sigma\in \Sigma\right\} & \left(\dfrac{\eta_{1,\kappa}}{p}\right)=1,\\[7pt]
\emptyset & \text{otherwise}.
\end{cases}
\]
Thus, by Corollary~\ref{cor:sufficient condition non-planarity},
we have the following result, 
which was observed in Section~11 of \cite{C2024}.

\begin{theorem}
\label{thm-n=1}
Let $\kappa\in\mathbb{Z}\setminus\{4\}$.
Then, there exists an explicit $K_{3,3}$-subdivision in $\Gkp$ 
if $\xi_{1,\kappa}\eta_{1,\kappa}\ne 0$ and $\left(\frac{\eta_{1,\kappa}}{p}\right)=1$.
\end{theorem}


\begin{remark}
\label{rem-n=1}
For 
\[
K_{(\alpha_{\kappa},-1)}
=\left(\begin{array}{ccc}
X_1,\!\! & X_2,\!\! & X_3 \\
Y_1,\!\! & Y_2,\!\! & Y_3
\end{array}\right)
=\left(\begin{array}{ccc}
 (\alpha_{\kappa},-1,-1),\!\! & (-1,\alpha_{\kappa},-1),\!\! & (-1,-1,\alpha_{\kappa}) \\
 (\overline{\alpha_{\kappa}},-1,-1),\!\! & (-1,\overline{\alpha_{\kappa}},-1),\!\! & (-1,-1,\overline{\alpha_{\kappa}})
 \end{array}\right)
\in \mathcal{K}^{\text{prop}}_{1,\kappa}(p),
\]
the $X_i$-$Y_j$-paths in $K_{(\alpha_{\kappa},-1)}$ 
for $1\le i\ne j\le 3$ are explicitly given as follows.
\begin{align*}
X_1
&=(\alpha_{\kappa},-1,-1)
 \overset{R_2}{\mapsto} (\alpha_{\kappa},\overline{\alpha_{\kappa}},-1)
 \overset{R_1}{\mapsto} (-1,\overline{\alpha_{\kappa}},-1)=Y_2,\\  
X_1
&=(\alpha_{\kappa},-1,-1)
 \overset{R_3}{\mapsto} (\alpha_{\kappa},-1,\overline{\alpha_{\kappa}})
 \overset{R_1}{\mapsto} (-1,-1,\overline{\alpha_{\kappa}})=Y_3,\\  
X_2
&=(-1,\alpha_{\kappa},-1)
 \overset{R_1}{\mapsto} (\overline{\alpha_{\kappa}},\alpha_{\kappa},-1)
 \overset{R_2}{\mapsto} (\overline{\alpha_{\kappa}},-1,-1)=Y_1,\\  
X_2
&=(-1,\alpha_{\kappa},-1)
 \overset{R_3}{\mapsto} (-1,\alpha_{\kappa},\overline{\alpha_{\kappa}})
 \overset{R_2}{\mapsto} (-1,-1,\overline{\alpha_{\kappa}})=Y_3,\\  
X_3
&=(-1,-1,\alpha_{\kappa})
 \overset{R_1}{\mapsto} (\overline{\alpha_{\kappa}},-1,\alpha_{\kappa})
 \overset{R_3}{\mapsto} (\overline{\alpha_{\kappa}},-1,-1)=Y_1,\\  
X_3
&=(-1,-1,\alpha_{\kappa})
 \overset{R_2}{\mapsto} (-1,\overline{\alpha_{\kappa}},\alpha_{\kappa})
 \overset{R_3}{\mapsto} (-1,\overline{\alpha_{\kappa}},-1)=Y_2.  
\end{align*}    
See Figure~\ref{fig:n=1} for explicit depictions of these paths in $K_{(\alpha_{\kappa},-1)}$ for certain $\kappa$ and $p$.
\end{remark}


\begin{figure}[H]
\begin{center}
\begin{tabular}{cc}
\begin{minipage}{0.45\linewidth}
\centering
\includegraphics[width=0.8\linewidth]{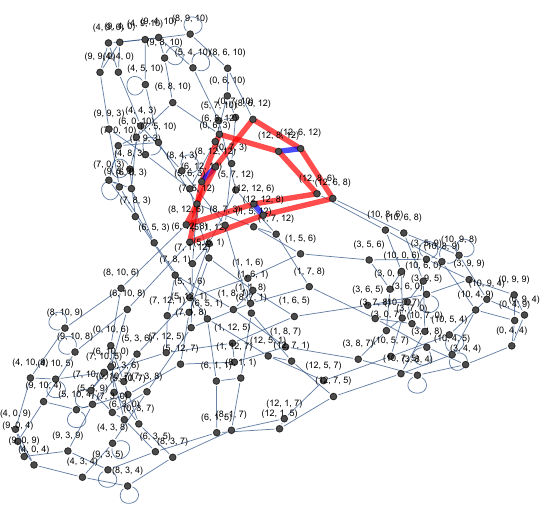}
\end{minipage}
\quad & \quad 
\begin{minipage}{0.45\linewidth}
\centering
\includegraphics[width=0.8\linewidth]{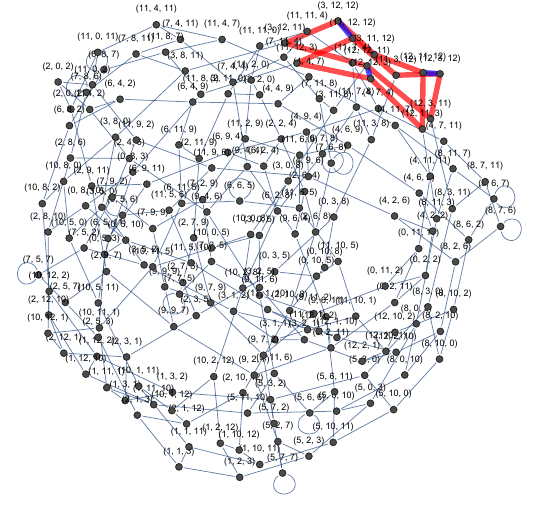}
\end{minipage}
\end{tabular}
\end{center}
\caption{Illustrations of $K_{(\alpha_{\kappa},-1)}\in \mathcal{K}^{\text{prop}}_{1,\kappa}(p)$ 
for $(\kappa,p)=(6,13)$ (left), and $(\kappa,p)=(8,13)$ (right).
The $X_i$-$Y_j$-paths (of length $2$)
and the $X_i$-$Y_i$ paths in $K_{(\alpha_{\kappa},-1)}$ 
are depicted using red and blue edges, respectively, 
for $1\le i,j\le 3$.
}
\label{fig:n=1}
\end{figure}

\subsection{The case $n=2$}

We next consider the case $n=2$.
Suppose that 
$\xi_{2,\kappa}=\kappa^2-13 \kappa+11\ne 0$ and 
$\eta_{2,\kappa}=16 \kappa^2-68 \kappa+41\ne 0$.
Then, any $(\alpha,\beta)\in T^{\text{dist}}_{2,\kappa}(p)$ 
must satisfy $B_{2,\kappa}(\alpha)=0$, where 
\[
 B_{2,\kappa}(y)=y^4-3y^3-(2 \kappa-7) y^2+(3 \kappa-4) y+\kappa^2-6\kappa+4.
\]
We first investigate the conditions under which the quartic equation
$B_{2,\kappa}(y)=0$ has a solution in $\Fp$.
It should be noted that, by Corollary~\ref{cor:all equal dist} (2),
the solutions of this equation appear in pairs; therefore, if a solution exists, there are either $2$ or $4$ solutions in $\Fp$.


Let $f(x)=x^4+ax^2+bx+c\in \Fp[x]$
be a quartic polynomial with discriminant  
\[
D(f(x))=-(4a^3+27b^2)b^2+16c(a^4+9ab^2-8a^2c+16c^2).
\]
We denote by $N_p(f(x))$ the number of solutions of $f(x)=0$ in $\Fp$, counted with multiplicity.
Let $\{S_n\}_{n\ge 0}$ be the sequence 
defined by the linear recurrence relation
\begin{align}
\label{def:sequence Sn}
\begin{split}
 S_0&=3, \quad S_1=-2a, \quad S_2=2a^2+8c, \\
  S_{n+3}&=-2aS_{n+2}+(4c-a^2)S_{n+1}+b^2S_n \quad (n\ge 0).    
\end{split}
\end{align}
Then, in the generic case where $(a^2+12c)bD(f(x))\ne 0$, 
the number of solutions $N_p(f(x))$ can be described in terms of $S_{n}$, as established in \cite[Theorems~5.5 and 5.7]{S2003}:
\begin{equation}
\label{for:Npfx}
 N_p(f(x))
=
\begin{cases}
4 & \text{$S_{p+1}=S_2$ and $S_{(p-1)/2}=3$},\\[3pt]
2 & \text{$S_{p+1}\ne a^2-4c,S_2$ and $\left(\dfrac{\mu_p}{p}\right)=1$ where 
$\mu_p=\dfrac{4a^3-16ac+9b^2-2aS_{p+1}}{-5a^2-12c+3S_{p+1}}$}, \\[3pt]
1 & S_{p+1}=a^2-4c,\\
0 & \text{otherwise}.
\end{cases}
\end{equation}
Notice that the condition $a^2+12c\ne 0$ implies that $S_2\ne a^2-4c$.



\begin{prop}
\label{prop:n=2}
Let $\kappa\in\mathbb{Z}\setminus\{4\}$.
Assume that 
$\xi_{2,\kappa}\eta_{2,\kappa}\ne 0$ 
and $5(\kappa^2-73\kappa+61)\ne 0$.
Then, we have 
\[
 N_p(B_{2,\kappa}(y))
=
\begin{cases}
4 & \text{$\left(\dfrac{\eta_{2,\kappa}}{p}\right)=1$ and $\left(\dfrac{5}{p}\right)
=\left(\dfrac{8\kappa-17+2\sqrt{\eta_{2,\kappa}}}{p}\right)
=\left(\dfrac{8\kappa-17-2\sqrt{\eta_{2,\kappa}}}{p}\right)=1$},\\[10pt]
2 & \text{$\left(\dfrac{\eta_{2,\kappa}}{p}\right)=-1$ and $\left(\dfrac{5}{p}\right)=1$}, \\[7pt]
0 & \text{otherwise}.
\end{cases}
\]
\end{prop}
\begin{proof}
Let
$f_{2,\kappa}(y)\coloneq B_{2,\kappa}\left(y+3/4\right)=y^4+ay^2+by+c$,
where 
$a=-(16\kappa-29)/8$, 
$b=25/8$ and 
$c=(256 \kappa^2-1248\kappa+1021)/256$.
Here, all rational coefficients are interpreted in $\Fp$. 
Since $a^2+12c=\kappa^2-73\kappa+61$ and 
$D(f_{2,\kappa}(y))=25(\kappa-4)^2\eta_{2,\kappa}$, 
our assumptions ensure that we are in the generic case 
in the sense described above.
Hence, we can apply the above results. 
In this case, the sequence $\{S_n\}_{n\ge 0}$ defined in \eqref{def:sequence Sn} is given by 
\begin{align*}
\begin{split}
 S_0&=3, \quad
 S_1=\frac{1}{4}(16\kappa-29), \quad
 S_2=\frac{1}{16}(16\kappa-49)(16\kappa-19), \\[5pt]
  S_{n+3}&
  =\frac{1}{4}(16\kappa-29)S_{n+2}
  -\frac{5}{16}(16\kappa-9)S_{n+1}
  +\frac{625}{64}S_n \quad (n\ge 0).
\end{split}
\end{align*}
Choose any $\sqrt{\eta_{2,\kappa}}\in\F_{p^2}$.
Then, we have  
\[
 S_n=
 \left(\frac{5}{4}\right)^n
 +\left(\frac{8\kappa-17+2\sqrt{\eta_{2,\kappa}}}{4}\right)^n
 +\left(\frac{8\kappa-17-2\sqrt{\eta_{2,\kappa}}}{4}\right)^n.
\]
\begin{itemize}
\item[(i)]
If $\sqrt{\eta_{2,\kappa}}\in \Fp$, 
then, since $(\sqrt{\eta_{2,\kappa}})^p=\sqrt{\eta_{2,\kappa}}$, we have 
\begin{align*}
 S_{p+1}
=\left(\frac{5}{4}\right)^2
 +\left(\frac{8\kappa-17+2\sqrt{\eta_{2,\kappa}}}{4}\right)^2
 +\left(\frac{8\kappa-17-2\sqrt{\eta_{2,\kappa}}}{4}\right)^2
=S_2.
\end{align*}
Moreover, from the Euler criterion, we have 
\begin{align*}
 S_{(p-1)/2}
&=\left(\frac{5}{4}\right)^{\frac{p-1}{2}}+\left(\frac{8\kappa-17+2\sqrt{\eta_{2,\kappa}}}{4}\right)^{\frac{p-1}{2}}+\left(\frac{8\kappa-17-2\sqrt{\eta_{2,\kappa}}}{4}\right)^{\frac{p-1}{2}}\\
&=5^{\frac{p-1}{2}}+\left(8\kappa-17+2\sqrt{\eta_{2,\kappa}}\right)^{\frac{p-1}{2}}+\left(8\kappa-17-2\sqrt{\eta_{2,\kappa}}\right)^{\frac{p-1}{2}}\\
&=\left(\frac{5}{p}\right)
+\left(\frac{8\kappa-17+2\sqrt{\eta_{2,\kappa}}}{p}\right)
+\left(\frac{8\kappa-17-2\sqrt{\eta_{2,\kappa}}}{p}\right).
\end{align*}
Therefore, 
$S_{(p-1)/2}=3$ if and only if 
$\left(\frac{5}{p}\right)
=\left(\frac{8\kappa-17+2\sqrt{\eta_{2,\kappa}}}{p}\right)
=\left(\frac{8\kappa-17-2\sqrt{\eta_{2,\kappa}}}{p}\right)=1$.
\item[(ii)] 
If $\sqrt{\eta_{2,\kappa}}\notin \Fp$,
then since $(\sqrt{\eta_{2,\kappa}})^p=-\sqrt{\eta_{2,\kappa}}$, we have 
\begin{align*}
 S_{p+1}
=\left(\frac{5}{4}\right)^2+2\left(\frac{8\kappa-17+2\sqrt{\eta_{2,\kappa}}}{4}\right)\left(\frac{8\kappa-17-2\sqrt{\eta_{2,\kappa}}}{4}\right) 
=\frac{275}{16}.
\end{align*}
Hence, 
$S_{p+1}\ne a^2-4c=5(16\kappa-9)/16$ since $\kappa\ne 4$, and 
$S_{p+1}\ne S_2$ since $\eta_{2,\kappa}\ne 0$. 
Moreover, a direct calculation shows that
\[
\mu_p=\frac{4a^3-16ac+9b^2-2aS_{p+1}}{-5a^2-12c+3S_{p+1}}
=\frac{5}{4}.
\]
Consequently, $\left(\frac{\mu_p}{p}\right)=1$ if and only if 
$\left(\frac{5}{p}\right)=1$.
\end{itemize}
This completes the proof from \eqref{for:Npfx}.
\end{proof}

From Propositions~\ref{prop:n=2} and \ref{porp:Kd=Kdd for n=1,2,3},
we immediately obtain the following theorem.

\begin{theorem}
\label{thm:n=2 non-planarity}
Let $\kappa\in\mathbb{Z}\setminus\{4\}$.
Then, there exists an explicit $K_{3,3}$-subdivision in $\Gkp$ if $\xi_{2,\kappa}\eta_{2,\kappa}\ne 0$, $5(\kappa^2-73\kappa+61)\ne 0$, 
and either of the following conditions holds. 
\begin{itemize}
\item[$\mathrm{(A)}$] 
$\left(\dfrac{\eta_{2,\kappa}}{p}\right)=1$ and 
$\left(\dfrac{5}{p}\right)
=\left(\dfrac{8\kappa-17+2\sqrt{\eta_{2,\kappa}}}{p}\right)
=\left(\dfrac{8\kappa-17-2\sqrt{\eta_{2,\kappa}}}{p}\right)=1$, or,
\\[0pt]
\item[$\mathrm{(B)}$] 
$\left(\dfrac{\eta_{2,\kappa}}{p}\right)=-1$ and $\left(\dfrac{5}{p}\right)=1$.
\end{itemize}
\end{theorem}

\begin{remark}
\label{rem-n=2}
It is clear that 
primes satisfying condition (B) in Theorem~\ref{thm:n=2 non-planarity} can be characterized by congruence relations modulo $5|\eta_{2,\kappa}|$, where $\eta_{2,\kappa}$ is regarded as an integer in only this context. 
On the other hand, 
primes satisfying condition (A) seems to be
characterized by their representation by certain quadratic forms with discriminant $5\eta_{2,\kappa}$ when it is not a square 
(for related results, see, e.g. \cite{SW1992}).
In the following, 
we provide lists of primes $p$ satisfying each condition for small $\kappa$. For convenience, we put
$W_{\kappa}\coloneq 5(\kappa^2-73\kappa+61)\xi_{2,\kappa}\eta_{2,\kappa} \in \F_p$.
\begin{itemize}
\item
When $\kappa=0$,
except for primes $p>3$ satisfying  
$W_0\ne 0$,
that is, $p\ne 5,11,41,61$,
they are 
\begin{itemize}
\item[(A)] 
$p=59$, $131$, $139$, $241$, $269$, $271$, $359$, $409$, $541$, $569$, $599$, $661$, $701$, $761$, $859$, $881$, $911$, $941$, $\ldots$.
Such a prime $p$ is represented by the reduced quadratic forms $x^2-13xy-9y^2$ or $5x^2-5xy-9y^2$ with discriminant $205$.
These two forms appear to represent the same set of integers.
\item[(B)]
$p\equiv \pm 6$, $\pm 11$, $\pm 14$, $\pm 19$, $\pm 24$, $\pm 26$, $\pm 29$, $\pm 34$, $\pm 44$, $\pm 54$, $\pm 56$, $\pm 69$, $\pm 71$, $\pm 76$, $\pm 79$, $\pm 89$, $\pm 94$, $\pm 96$, $\pm 99$, $\pm 101$ 
 $\pmod{205}$.
\end{itemize} 
\item
When $\kappa=1$,
except for primes $p>3$ satisfying  
$W_1\ne 0$, that is, $p\ne 5,11$,
they are 
\begin{itemize}
\item[(A)] 
$p=59$, $71$, $199$, $229$, $251$, $269$, $311$, $379$, $389$, $499$, $509$, $631$, $661$, $691$, $751$, $839$, $881$, $929$, $\ldots$.
Such a prime $p$ is represented by the reduced quadratic form $x^2+xy+14y^2$ with discriminant $-55$.
\item[(B)]
$p\equiv 6$, $19$, $21$, $24$, $29$, $39$, $41$, $46$, $51$, $54$ $\pmod{55}$.
\end{itemize} 
\item
When $\kappa=2$,
except for primes $p>3$ satisfying  
$W_2\ne 0$, that is, $p\ne 5,11,31$,
they are 
\begin{itemize}
\item[(A)] 
$p=41$, $59$, $149$, $191$, $311$, $349$, $379$, $419$, $421$, $439$, $479$, $691$, $701$, $769$, $811$, $971$, $\ldots$.
Such a prime $p$ is represented by the reduced quadratic form $x^2+xy+39y^2$ with discriminant $-155$.
\item[(B)]
$p\equiv 6$, $11$, $21$, $24$, $26$, $29$, $34$, $44$, $46$, $54$, $61$, $74$, $79$, $84$, $86$, $89$, $91$, 
$96$, $99$, $104$, $106$, $114$, $116$, $119$, $136$, $139$, $141$, $146$, $151$, $154$ $\pmod{155}$.
\end{itemize} 
\item
When $\kappa=3$,
except for primes $p>3$ satisfying  
$W_3\ne 0$, that is, $p\ne 5,19,149$,
they are 
\begin{itemize}
\item[(A)] 
$p=19$, $131$, $191$, $199$, $239$, $251$, $349$, $389$, $419$, $461$, $491$, $709$, $739$, $809$, $821$, $859$, $919$, $929$, $\ldots$.
Such a prime $p$ is represented by the reduced quadratic forms $x^2+xy+24y^2$ or $5x^2+5xy+6y^2$ with discriminants $-95$.
\item[(B)]
$p\equiv 14$, $21$, $29$, $31$, $34$, $41$, $46$, $51$, $56$, $59$, $69$, $71$, $79$, $84$, $86$, $89$, $91$, $94$ $\pmod{95}$.
\end{itemize} 
\end{itemize}
\end{remark}

{\small 
\begin{table}[!ht]
\label{table:T2}
\begin{center}
\begin{tabular}{|c||c|c|c|c|}
\hline
$p$ & $T^{\text{dist}}_{2,0}(p)$ & $T^{\text{dist}}_{2,1}(p)$ & $T^{\text{dist}}_{2,2}(p)$ & $T^{\text{dist}}_{2,3}(p)$ \\
\hline
\hline
    $5$  & $\ast\ast\ast$
    & $\ast\ast\ast$
    & $\ast\ast\ast$
    & $\ast\ast\ast$
    \\
    \hline
    $7$  & $\emptyset$ & $\emptyset$ & $\emptyset$ & $\emptyset$ \\
    \hline
    $11$ & 
    $\ast\ast\ast$
    & 
    $\ast\ast\ast$
    & 
    $\ast\ast\ast$
    & $\emptyset$ \\
    \hline
    $13$ & $\emptyset$ & $\emptyset$ & $\emptyset$ & $\emptyset$ \\
    \hline
    $17$ & $\emptyset$ & $\emptyset$ & $\emptyset$ & $\emptyset$ \\
    \hline
    $19$ & $\{(7,14), (18,14)\}$ & $\{(9,14), (16,14)\}$ & $\emptyset$ &
    $\ast\ast\ast$
    \\
    \hline
    $23$ & $\emptyset$ & $\emptyset$ & $\emptyset$ & $\emptyset$ \\
    \hline
    $29$ & $\{(11,23), (25,23)\}$ & $\{(8,5), (17,5)\}$ & $\{(3,23), (4,23)\}$ & $\{(9,23), (27,23)\}$ \\
    \hline
    $31$ & $\emptyset$ & $\emptyset$ & 
    $\ast\ast\ast$ 
    & $\{(9,12), (11,12)\}$ \\
    \hline
    $37$ & $\emptyset$ & $\emptyset$ & $\emptyset$ & $\emptyset$ \\
    \hline
    $41$ & 
    $\ast\ast\ast$
    & $\{(3,34), (5,34)\}$ & 
    \begin{tabular}{c}
$\{(11,6), (25,6),$ \\
\ \ $(21,34), (28,34)\}$
\end{tabular} 
     & $\{(15,6), (21,6)\}$ \\
    \hline
    $43$ & $\emptyset$ & $\emptyset$ & $\emptyset$ & $\emptyset$ \\
    \hline
    $47$ & $\emptyset$ & $\emptyset$ & $\emptyset$ & $\emptyset$ \\
    \hline
\end{tabular}
\caption{The sets $T^{\text{dist}}_{2,\kappa}(p)$ for $\kappa=0,1,2,3$. 
Here, the symbol $\ast\ast\ast$ indicates that
$W_\kappa\coloneq 5(\kappa^2-73\kappa+61)\xi_{2,\kappa}\eta_{2,\kappa}=0$.}
\end{center}
\end{table}
}

\begin{remark}
For $K_{(\alpha,\beta)}
=\left(\begin{array}{ccc}
X_1,\!\! & X_2,\!\! & X_3 \\
Y_1,\!\! & Y_2,\!\! & Y_3
\end{array}\right)
\in \mathcal{K}^{\text{prop}}_{2,\kappa}(p)$,
the $X_i$-$Y_j$-paths in $K_{(\alpha,\beta)}$ 
for $1 \le i\ne j\le 3$ are explicitly given as follows.
\begin{align*}
X_1
&=(\alpha,\beta,\beta)
 \overset{R_2}{\mapsto}
 (\alpha,\beta(\alpha-1),\beta)
 \overset{R_1}{\mapsto} (\beta(\overline{\alpha}-1),\beta(\alpha-1),\beta) \overset{R_2}{\mapsto} (\beta(\overline{\alpha}-1),\overline{\alpha},\beta) 
 \overset{R_1}{\mapsto}
 (\beta,\overline{\alpha},\beta)=Y_2,\\  
X_1
&=(\alpha,\beta,\beta)
 \overset{R_3}{\mapsto}
 (\alpha,\beta,\beta(\alpha-1))
 \overset{R_1}{\mapsto} (\beta(\overline{\alpha}-1),\beta,\beta(\alpha-1)) \overset{R_3}{\mapsto} (\beta(\overline{\alpha}-1),\beta,\overline{\alpha}) 
 \overset{R_1}{\mapsto}
 (\beta,\beta,\overline{\alpha})=Y_3,\\  
X_2
&=(\beta,\alpha,\beta)
 \overset{R_1}{\mapsto}
 (\beta(\alpha-1),\alpha,\beta)
 \overset{R_2}{\mapsto}
 (\beta(\alpha-1),\beta(\overline{\alpha}-1),\beta)
 \overset{R_1}{\mapsto} (\overline{\alpha},\beta(\overline{\alpha}-1),\beta,) 
 \overset{R_2}{\mapsto}
 (\overline{\alpha},\beta,\beta)=Y_1,\\  
X_2
&=(\beta,\alpha,\beta)
 \overset{R_3}{\mapsto}
 (\beta,\alpha,\beta(\alpha-1))
 \overset{R_2}{\mapsto} (\beta,\beta(\overline{\alpha}-1),\beta(\alpha-1)) \overset{R_3}{\mapsto} (\beta,\beta(\overline{\alpha}-1),\overline{\alpha}) 
 \overset{R_2}{\mapsto}
 (\beta,\beta,\overline{\alpha})=Y_3,\\  
X_3
&=(\beta,\beta,\alpha)
 \overset{R_1}{\mapsto}
 (\beta(\alpha-1),\beta,\alpha)
 \overset{R_3}{\mapsto}
 (\beta(\alpha-1),\beta,\beta(\overline{\alpha}-1))
 \overset{R_1}{\mapsto} (\overline{\alpha},\beta,\beta(\overline{\alpha}-1)) 
 \overset{R_3}{\mapsto}
 (\overline{\alpha},\beta,\beta)=Y_1,\\  
X_3
&=(\beta,\beta,\alpha)
 \overset{R_2}{\mapsto}
 (\beta,\beta(\alpha-1),\alpha)
 \overset{R_3}{\mapsto}
 (\beta,\beta(\alpha-1),\beta(\overline{\alpha}-1))
 \overset{R_2}{\mapsto} (\beta,\overline{\alpha},\beta(\overline{\alpha}-1))
 \overset{R_3}{\mapsto}
 (\beta,\overline{\alpha},\beta)=Y_2.  
\end{align*}    
See Figure~\ref{fig:n=2} for explicit depictions of these paths in $K_{(\alpha,\beta)}$ for certain $\kappa$ and $p$.
\end{remark}

\begin{figure}[H]
\begin{center}
\begin{tabular}{cc}
\begin{minipage}{0.45\linewidth}
\centering
\includegraphics[width=0.8\linewidth]{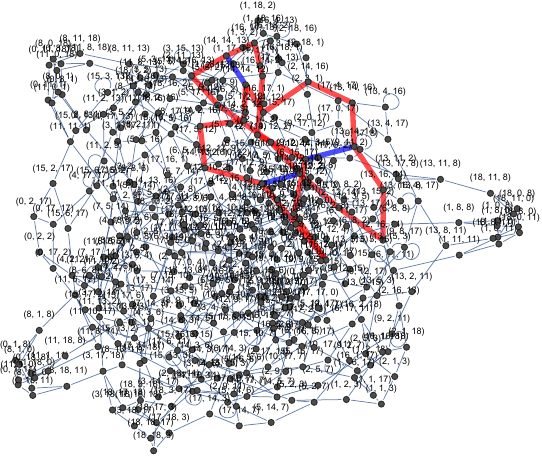}
\end{minipage}
\quad & \quad 
\begin{minipage}{0.45\linewidth}
\centering
\includegraphics[width=0.8\linewidth]{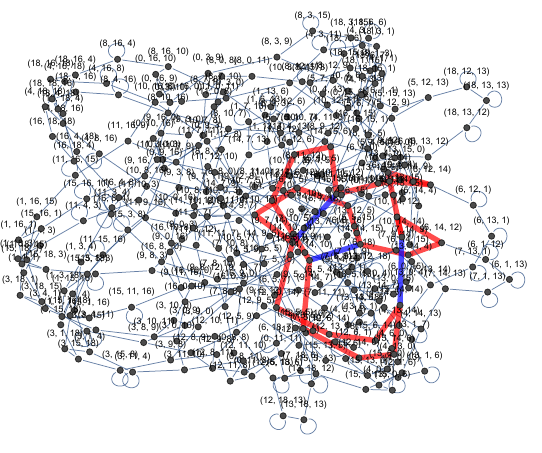}
\end{minipage}
\end{tabular}
\end{center}
\caption{Illustrations of $K_{(\alpha,\beta)}\in \mathcal{K}^{\text{prop}}_{2,\kappa}(p)$ 
for $(\kappa,p)=(8,19)$ with $(\alpha,\beta)=(12,14)$ (left),
and $(\kappa,p)=(-5,19)$ with $(\alpha,\beta)=(10,14)$ (right).
The $X_i$-$Y_j$-paths (of length $4$)
and the $X_i$-$Y_i$ paths in $K_{(\alpha,\beta)}$ 
are depicted using red and blue edges, respectively,
for $1\le i,j\le 3$.
}
\label{fig:n=2}
\end{figure}

\subsection{The case $n\ge 3$}

To apply our strategy to the case $n\ge 3$, 
we need to solve the algebraic equation $B_{n,\kappa}(y)=0$ of degree $2n$ in $\Fp$, which is generally difficult.
For $n=3$, 
we provide a table of $T^{\text{dist}}_{3,\kappa}(p)$ below.
Each pair $(\alpha,\beta)$ in this table indeed yields a $K_{3,3}$-subdivision $K_{(\alpha,\beta)}\in\mathcal{K}^{\text{prop}}_{3,\kappa}(p)$ in $\Gkp$ by Proposition~\ref{porp:Kd=Kdd for n=1,2,3} (3),
provided the specified conditions are satisfied.
See Figure~\ref{fig:n=3} for explicit depictions of 
the $X_i$-$Y_i$ paths in $K_{(\alpha,\beta)}$ for $1\le i,j\le 3$
for certain $\kappa$ and $p$.
 
{\small 
\begin{table}[!ht]
\label{table:T3}
\begin{center}
\begin{tabular}{|c||c|c|c|c|}
\hline
$p$ & $T^{\text{dist}}_{3,0}(p)$ & $T^{\text{dist}}_{3,1}(p)$ & $T^{\text{dist}}_{3,2}(p)$ & $T^{\text{dist}}_{3,3}(p)$ \\
\hline
\hline
$5$ & $\emptyset$ & $\emptyset$ & $\emptyset$ & $\emptyset$ \\
\hline
$7$ & $\emptyset$ & $\{(0, 2), (4, 2)\}$ & $\emptyset$ & $\emptyset$ \\
\hline
$11$ & $\emptyset$ & $\emptyset$ & $\emptyset$ & $\emptyset$ \\
\hline
$13$ & 
\begin{tabular}{c}
$\{(1, 8), (11, 8),$ \\
\ \ \ \ $(3, 10), (6, 10)\}$
\end{tabular} 
& $\emptyset$ & $\{(5, 8), (7, 8)\}$ & $\emptyset$ \\
\hline
$17$ & $\emptyset$ & $\emptyset$ & $\emptyset$ & $\emptyset$ \\
\hline
$19$ & $\emptyset$ & $\emptyset$ & $\emptyset$ & $\emptyset$ \\
\hline
$23$ & $\emptyset$ & $\emptyset$ & $\emptyset$ & $\emptyset$ \\
\hline
$29$ & $\emptyset$ & $\{(14, 3), (24, 3)\}$ & 
\begin{tabular}{c}
$\{(8, 7), (12, 7),$ \\
\ \ \ $(10, 18), (24, 18)\}$
\end{tabular} 
 & $\{(21, 7), (28, 7)\}$ \\
\hline
$31$ & $\emptyset$ & $\emptyset$ & $\emptyset$ & $\emptyset$ \\
\hline
$37$ & $\emptyset$ & $\emptyset$ & $\emptyset$ & $\emptyset$ \\
\hline
$41$ & 
\begin{tabular}{c}
$\{(12, 30), (27, 30),$ \\
\ \ \ $(22, 37), (35, 37)\}$
\end{tabular} 
 & $\{(27, 37), (30, 37)\}$ & $\emptyset$ & 
 \begin{tabular}{c}
$\{(6, 30), (33, 30),$ \\
\ \ \ $(36, 14), (37, 14)\}$
\end{tabular} 
 \\
\hline
$43$ & 
\begin{tabular}{c}
$\{(23, 19), (37, 19),$ \\
\ \ \ $(26, 8), (38, 8)\}$
\end{tabular} 
 & $\{(12, 15), (41, 15)\}$ & $\{(23, 8), (41, 8)\}$ & $\{(27, 19), (33, 19)\}$ \\
\hline
$47$ & $\emptyset$ & $\emptyset$ & $\emptyset$ & $\emptyset$ \\
\hline
\end{tabular}
\caption{The sets $T^{\text{dist}}_{3,\kappa}(p)$ for $\kappa=0,1,2,3$.}
\end{center}
\end{table}
}

\begin{figure}[H]
\begin{center}
\begin{tabular}{cc}
\begin{minipage}{0.45\linewidth}
\centering
\includegraphics[width=0.8\linewidth]{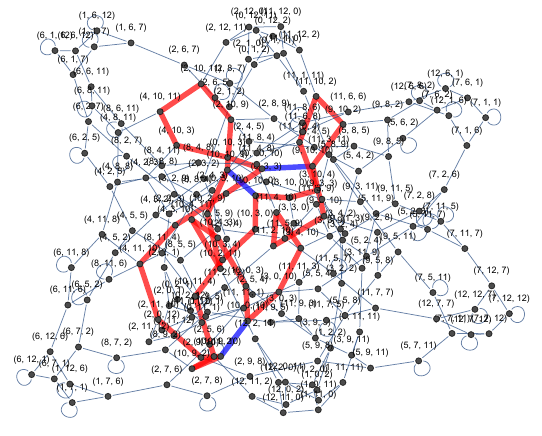}
\end{minipage}
\quad & \quad 
\begin{minipage}{0.45\linewidth}
\centering
\includegraphics[width=0.8\linewidth]{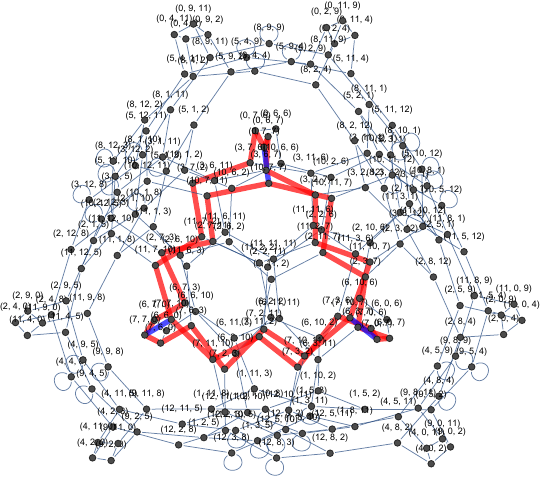}
\end{minipage}
\end{tabular}
\end{center}
\caption{Illustrations of $K_{(\alpha,\beta)}\in\mathcal{K}^{\text{prop}}_{3,\kappa}(p)$ 
for $(\kappa,p)=(5,13)$ with $(\alpha,\beta)=(0,10)$ (left),
and $(\kappa,p)=(7,13)$ with $(\alpha,\beta)=(0,7)$ (right).
The $X_i$-$Y_j$-paths (of length $6$)
and the $X_i$-$Y_i$ paths in $K_{(\alpha,\beta)}$ 
are depicted using red and blue edges, respectively, 
for $1\le i,j\le 3$.}
\label{fig:n=3}
\end{figure}



\subsection{The case $n=(p-1)/2$}

We finally consider the case $n=(p-1)/2$
by employing an approach different from the one used in the previous sections.
This yields the following result, 
which generalizes \cite[Theorem~1.4]{C2024}.

\begin{theorem}
\label{thm-n=(p-1)/2}
Let $\kappa\in\mathbb{Z}\setminus\{4\}$.
Then, there exists an explicit $K_{3,3}$-subdivision in $\Gkp$ 
if $\left(\frac{\kappa-4}{p}\right)=1$.
\end{theorem}
\begin{proof}
When $n=(p-1)/2$, which implies $2n+1\equiv 0 \pmod{p}$,
Lemma~\ref{lem:Am properties} (3) ensures that 
$\beta=2$ is a solution of $A_n(x/2)=0$.  
Under our assumption,
we can take $(\alpha,\beta)=(2+\sqrt{\kappa-4},2)\in\Tnkpdist$
and hence 
\[
K_{(\alpha,\beta)}=
\left(\begin{array}{ccc}
(2+\sqrt{\kappa-4},2,2),\!\! & (2,2+\sqrt{\kappa-4},2),\!\! & (2,2,2+\sqrt{\kappa-4}) \\
(2-\sqrt{\kappa-4},2,2),\!\! & (2,2-\sqrt{\kappa-4},2),\!\! & (2,2,2-\sqrt{\kappa-4})
\end{array}\right)
\in \Knkpdist.
\]
Now, we prove that $K_{(\alpha,\beta)}\in\Knkpprop$.
As discussed in Section~\ref{sec:K33 minors},
it suffices to show that 
$X_1=(2+\sqrt{\kappa-4},2,2)$ and  
$X_2=(2,2+\sqrt{\kappa-4},2)$ do not appear as internal vertices 
of the $X_1$-$Y_2$-path.
By induction on $m$, we have   
\begin{align*}
 (R_1R_2)^m(X_1)
&=(2+(2m+1)\sqrt{\kappa-4},2+2m\sqrt{\kappa-4},2),\\
 (R_2(R_1R_2)^m)(X_1)
&=(2+(2m+1)\sqrt{\kappa-4},2+2(m+1)\sqrt{\kappa-4},2).
\end{align*}
From these explicit forms,  
we see that $(R_1R_2)^m(X_1)=X_1$ holds if and only if $p\,|\,m$, 
and $(R_2(R_1R_2)^m)(X_1)=X_1$ never occur.
Similarly, $(R_1R_2)^m(X_1)=X_2$ never occurs, 
whereas $(R_2(R_1R_2)^m)(X_1)=X_2$ holds
if and only if $p\,|\,(2m+1)$.
Since neither $p\,|\,m$ for $0<m<n$ nor 
$p\,|\,(2m+1)$ for $0\le m < n$ holds
as both $m$ and $2m+1$ are strictly less than $2n+1 = p$, 
the desired claim follows.
\end{proof}

\begin{figure}[H]
\begin{center}
\begin{tabular}{cc}
\begin{minipage}{0.45\linewidth}
\centering
\includegraphics[width=0.8\linewidth]{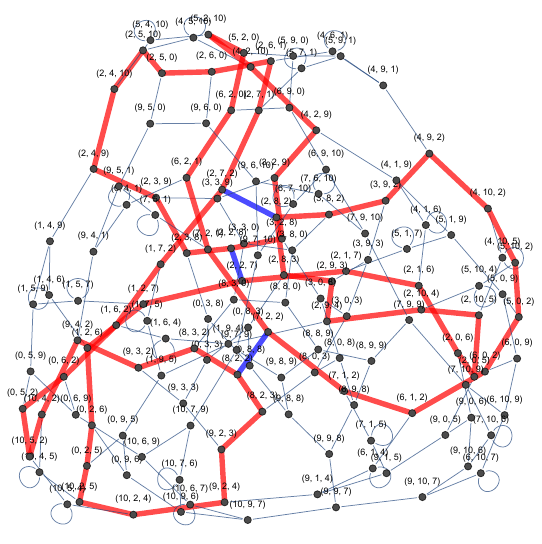}
\end{minipage}
\quad & \quad 
\begin{minipage}{0.45\linewidth}
\centering
\includegraphics[width=0.8\linewidth]{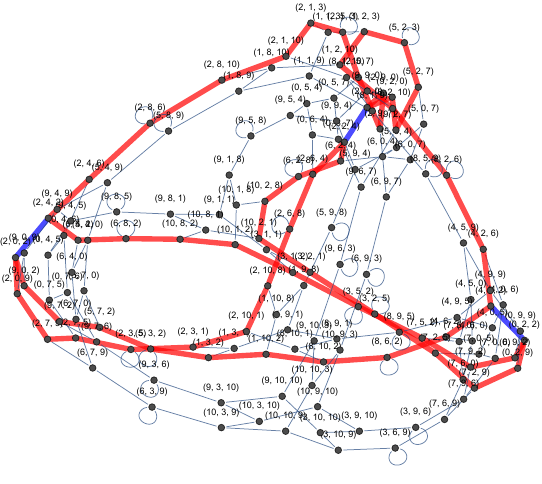}
\end{minipage}
\end{tabular}
\end{center}
\caption{Illustrations of $K_{(\alpha,2)}\in\mathcal{K}^{\mathrm{prop}}_{(p-1)/2,\kappa}(p)$ 
for $(\kappa,p)=(7,11)$ (left), and $(\kappa,p)=(8,11)$ (right)
with $\alpha=2+\sqrt{\kappa-4}$.
The $X_i$-$Y_j$-paths (of length $p-1$)
and the $X_i$-$Y_i$ paths in $K_{(\alpha,2)}$ 
are depicted using red and blue edges, respectively,
for $1\le i,j\le 3$.}
\label{fig:n=p-1 div 2}
\end{figure}

\subsection{Natural densities}
\label{subsec:density}
Let $\kappa\in\mathbb{Z}\setminus\{4\}$.
From Theorem~\ref{thm-n=1}, Theorem~\ref{thm:n=2 non-planarity} (B),
or Theorem~\ref{thm-n=(p-1)/2}, 
it follows from Dirichlet's theorem on arithmetic progressions that there are infinitely many primes $p$ for which
$\Gkp$ contains a $K_{3,3}$-subdivision, which yields topological properties of $\Gkp$ such as the non-planarity and existence of short cycles by Corollaries~\ref{cor:sufficient condition non-planarity} and \ref{cor:cycles}, respectively.
This motivates us to consider the natural density of such primes $p$: 
\[
\delta_{\kappa}
\coloneq\lim_{x\to\infty}
\frac{\#\{p\le x\,|\,\text{$\Gkp$ contains a $K_{3,3}$-subdivision}\}}{\#\{p\le x\}}.
\]
It holds that $\delta_0=1$
since $\mathcal{G}(p)$ is non-planar for any $p\neq 7$
by \cite[p.115]{C2024}.
For any $\kappa \neq 4$, Theorem~\ref{thm-n=1} and the Chebotarev density theorem (see, e.g.~\cite{N1999}) 
imply that $ \delta_{\kappa}\geq 1/2$.
Moreover, combining the three theorems above,
we arrive at the following theorem.

\begin{theorem}
\label{thm:np density}
\begin{itemize}
\item[$(1)$]
For each $\kappa\in\mathbb{Z}\setminus\{4\}$ such that none of  
$\eta_{1,\kappa}$,  
$\eta_{2,\kappa}$, 
$5$, and $\kappa-4$
is a square in $\mathbb{Z}$, 
and that the ratio of any two of them is not a square in $\mathbb{Q}$, 
we have
\begin{equation}
\label{for:lower bound of Dnp}
\delta_{\kappa}\ge \frac{13}{16}.
\end{equation}
That is, $\Gkp$ contains a $K_{3,3}$-subdivision
for at least $81.25\%$ 
of primes $p$ asymptotically.
\item[$(2)$]
The natural density 
of $\kappa\in\mathbb{Z}\setminus\{4\}$ for which $(1)$ holds 
is equal to $1$.
\end{itemize}
\end{theorem}

\begin{remark}
If either $\eta_{1,\kappa}$ or $\kappa-4$ is a square in $\mathbb{Z}$,
then $\Gkp$ is non-planar for all but finitely many primes $p$. 
Consequently, $\delta_{\kappa}=1$ for such $\kappa\in\mathbb{Z}\setminus\{4\}$.
Note that, one can check that $\eta_{1,\kappa}$ is a square if $\kappa=t^2+t+2$ for some ineteger $t$.
On the other hand, if $\eta_{2,\kappa}$ is a square in $\mathbb{Z}$, 
which occurs if and only if $\kappa=10$,
then, since the conditions in
Theorem~\ref{thm:n=2 non-planarity} (A) 
reduces to $(\frac{5}{p})=1$,
the Chebotarev density theorem implies that
$\delta_{\kappa}\ge 1/2$.

\end{remark}

\begin{proof}[Proof of Theorem~\ref{thm:np density}]
We first prove $(1)$.
For $a_1,\ldots,a_r\in\mathbb{Z}$
and $\varepsilon_1,\ldots,\varepsilon_r\in\{\pm 1\}$, define 
\[
Q(a_1,\ldots,a_r;\varepsilon_1,\ldots,\varepsilon_r)
\coloneq 
\lim_{x\to\infty}\frac{
\#\{p\le x\,|\,(\frac{a_1}{p})=\varepsilon_1,\ldots,(\frac{a_r}{p})=\varepsilon_r\}}{\#\{p\le x\}}.
\]
By the Chebotarev density theorem and the prime number theorem, 
if none of $a_1,\dots,a_r$ is a square in $\mathbb{Z}$ and 
the ratio of any two of them is also not a square in $\mathbb{Q}$, 
then
\begin{equation}
\label{for:Chebotarev}
Q(a_1,\ldots,a_r;\varepsilon_1,\ldots,\varepsilon_r)
=\frac{1}{2^r}.
\end{equation}
Notice that $2^r$ is the order of the Galois group of the multiquadratic field $\mathbb{Q}(\sqrt{a_1},\ldots,\sqrt{a_r})$ over $\mathbb{Q}$ 
under our assumptions.
Then, by Theorem~\ref{thm-n=1}, Theorem~\ref{thm:n=2 non-planarity} (B), and Theorem~\ref{thm-n=(p-1)/2}, 
and applying the inclusion-exclusion principle, 
we have 
\begin{align*}
\delta_{\kappa}
&\ge Q(\eta_{1,\kappa};1)+Q(\eta_{2,\kappa},5;-1,1)+Q(\kappa-4;1)\\
&\quad -Q(\eta_{1,\kappa},\eta_{2,\kappa},5;1,-1,1)
-Q(\eta_{2,\kappa},5,\kappa-4;-1,1,1)-Q(\kappa-4,\eta_{1,\kappa};1,1)\\
&\quad +Q(\eta_{1,\kappa},\eta_{2,\kappa},5,\kappa-4;1,-1,1,1)\\
&=\frac{13}{16}.
\end{align*}

We next show $(2)$. 
For $\kappa\in\mathbb{Z}\setminus\{4\}$, put  
$b_1(\kappa)\coloneq 5$, 
$b_2(\kappa)\coloneq \kappa-4$, 
$b_3(\kappa)\coloneq \eta_{1,\kappa}=4\kappa-7$, and $b_4(\kappa)\coloneq\eta_{2,\kappa}=16\kappa^2-68\kappa+41$.
Moreover, define 
\begin{align*}
u_i(x)
&\coloneq \#\{\kappa\in\mathbb{Z}\setminus\{4\}\,|\,\text{$|\kappa|\le x$, $b_i(\kappa)$ is a square in $\mathbb{Z}$}\}
\quad (1\le i\le 4),\\
v_{ij}(x)
&\coloneq \#\{\kappa\in\mathbb{Z}\setminus\{4\}\,|\,\text{$|\kappa|\le x$, $b_j(\kappa)/b_i(\kappa)$ is a square in $\mathbb{Q}$}\}
\quad (1\le i<j\le 4).
\end{align*}
To obtain the desired result,
it suffices to show that
$u_i(x)=o(x)$ for $1\le i\le 4$,
and $v_{ij}(x)=o(x)$ for $1\le i<j\le 4$ as $x\to \infty$.

It is easy to see that $u_1(x)=0$ and $u_2(x)=u_3(x)=O(\sqrt{x})$.
For $u_4(x)$, we observe that  
if $b_4(\kappa)=m^2$ for some $m\in\mathbb{Z}$, 
then $(8\kappa-17+2m)(8\kappa-17-2m)=125$,
which implies that $\kappa$ must be equal to $4$ or $10$.
This shows that $u_4(x)=O(1)$.

For $v_{12}(x)$, observe that 
if $b_2(\kappa)/b_1(\kappa)$ is a square in $\mathbb{Q}$,
then it must also be a square in $\mathbb{Z}$.
Hence, by the same discussion as above, 
we have $v_{12}(x)=O(\sqrt{x})$.
Similarly, $v_{13}(x)=O(\sqrt{x})$.
For $v_{14}(x)$, suppose that $b_4(\kappa)/b_1(\kappa)$ is a square in $\mathbb{Q}$.
Then it is also be a square in $\mathbb{Z}$,
and we can write $b_4(\kappa)/b_1(\kappa)=m^2$ for some $m\in\mathbb{Z}$, or equivalently, 
$16\kappa^2-68\kappa+41-5m^2=0$.
The discriminant of this quadratic equation in $\kappa$, divided by $16$, is $20m^2+125$, 
which must be a square in $\mathbb{Z}$.
This means that $m$ satisfies the Pell-type equation $s^2-20m^2=125$,
and by the standard theory of such equations, 
the number of solutions $(s,m)$ with $|m| \le x$ grows logarithmically; thus, $v_{14}(x)=O(\log{x})$.
For $v_{23}(x)$, assume that  
there exist $a,b\in\mathbb{Z}$ with $\gcd(a,b)=1$
such that $b_3(\kappa)/b_2(\kappa)=a^2/b^2$, i.e.
$\kappa=(7b^2-4a^2)/(4b^2-a^2)$.
Put $s= 4b^2-a^2=(2b+a)(2b-a)$.
Then we see that $s\,|\,9a^2$ and $s\,|\,9b^2$, 
hence $s\,|\,\gcd(9a^2,9b^2)=9$, and thus $s\in\{\pm 1,\pm 3,\pm 9\}$.
This restricts $(a, b)$ to the finite set 
$(a,b)\in\{(\pm 1,\pm 1),(\pm 5,\pm 2)\}$, 
meaning $\kappa$ must be $1$ or $8$. 
Hence, $v_{23}(x)=O(1)$.
For $v_{24}(x)$, let $g=\gcd(b_4(\kappa),b_2(\kappa))$.
A simple calculation shows $g\in\{1,5,25\}$.
Assume that $b_4(\kappa)/b_2(\kappa)$ is a square in $\mathbb{Q}$.
Then, there exist $a,b\in\mathbb{Z}$ with $\gcd(a,b)=1$ such that $b_4(\kappa)=ga^2$ and $b_2(\kappa)=gb^2$.
Notice that, from these equations, 
for given $a$ and $g$, 
there are only finitely many possible $b$.
Now, from the former equation, 
the discriminant of the resulting quadratic equation in $\kappa$, 
divided by $16$, is equal to $4ga^2+125$ and must be a square in $\mathbb{Z}$.
Therefore, $a$ satisfies the Pell-type equation $s^2-4ga^2=125$.
Hence, by applying again the standard theory of this type of equation, together with the above observation, we obtain $v_{24}(x)=O(\log{x})$.
Finally, by the same argument, we have $v_{34}(x)=O(\log{x})$.

These complete the proof.
\end{proof}

The non-negative integers $\kappa$ that satisfy the conditions in Theorem~\ref{thm:np density} (1) are listed as follows:
\[
0, 6, 7, 11, 12, 15, 16, 17, 18, 19, 
21, 23, 25, 26, 27, 28, 30, 31, 34, 35, 36, 37, 38, 39, \ldots.
\]

\begin{remark}
We might improve the lower bound \eqref{for:lower bound of Dnp}
by taking into account primes that satisfy the conditions
in Theorem~\ref{thm:n=2 non-planarity} (A).
Indeed, since these conditions are equivalent to the quartic equation $B_{2,\kappa}(y)=0$ having four solutions in $\mathbb{F}_p$, for any $\kappa\in\mathbb{Z}\setminus\{4\}$, the Chebotarev density theorem implies that the corresponding natural density is bounded below by $1/24$, 
where $24$ is the order of the symmetric group of degree $4$,
the maximal possible Galois group of $B_{2,\kappa}(y)$.
However, unlike the case (B), 
it is difficult to control the overlap between the set of primes satisfying the conditions in Theorem~\ref{thm:n=2 non-planarity} (A) 
and those appearing in Theorem~\ref{thm-n=1} or Theorem~\ref{thm-n=(p-1)/2}.
Due to this potential intersection, 
we cannot currently provide an improved explicit lower bound 
for $\delta_{\kappa}$.

\end{remark}

\begin{remark}
For general $\kappa \in \mathbb{Z}$, 
it would be possible to prove that $\delta_\kappa=1$
if one could evaluate the asymptotic behavior 
of the Euler characteristic of $\Gkp$ similar to the case of $\kappa=0$.
\end{remark}

\section{Mutually vertex-disjoint $K_{3,3}$-subdivisions}
\label{sect:disjoint}


In this section, 
we construct mutually vertex-disjoint $K_{3,3}$-subdivisions in $\Gkp$. 

\begin{theorem}
\label{thm-disjoint-k}
\begin{enumerate}
\item[$(1)$] 
For each $\kappa\in\mathbb{Z}\setminus\{4\}$,
$\Gkp$ contains at least four mutually vertex-disjoint 
$K_{3,3}$-subdivisions for infinitely many primes $p$ of natural density at least $1/4$ if $\kappa=2$ and $1/2$ otherwise. 
Furthermore, if $\kappa$ is an integer such that 
none of $\eta_{1, \kappa}$, $\eta_{2,\kappa}$ and $5$ is a square in $\mathbb{Z}$ and that the ratio of any two of them is also not a square in $\mathbb{Q}$,
then the natural density of such primes 
is at least $5/8$.
\item[$(2)$]
The natural density of $\kappa\in\mathbb{Z}\setminus\{4\}$
for which the latter assertion of $(1)$ holds is equal to $1$.
\end{enumerate}
\end{theorem}

\begin{remark}
\label{rem-k=2}
Let us consider the $K_{3,3}$-subdivision $K_{(\alpha_{\kappa},-1)} \in \mathcal{K}^{\mathrm{prop}}_{1,\kappa}(p)$ appearing in Section~\ref{subsec:n=1}.
When $\kappa=2$, for any prime $p$,
$\mathcal{G}_{2}(p)$ contains a specific component $\mathcal{S}_{2}(p)$ consisting of $16$ vertices, given by the orbit of $(1,1,1)\in\mathcal{M}_2(p)$ under the group generated by Vieta involutions 
(this component is unique according to \cite[Theorem~1.1]{Mart2025}).
We observe that $K_{(\alpha_{2},-1)}=K_{(1,-1)} \in \mathcal{K}^{\mathrm{prop}}_{1,2}(p)$ and its images $\sigma\bigl(K_{(\alpha_{2},-1)}\bigr)$ under any double sign change $\sigma$ cannot be mutually vertex-disjoint since these are contained within $\mathcal{S}_{2}(p)$ (see Figure~\ref{fig:k2smallcomponent}).
On the other hand, when $\kappa\neq 2$, $K_{(\alpha_{\kappa},-1)}$ and its copies by double sign changes are mutually vertex-disjoint as explained in the proof of Theorem~\ref{thm-disjoint-k}. 

\begin{figure}[H]
\begin{center}
\includegraphics[width=0.9\linewidth, height=0.2\textheight, keepaspectratio]{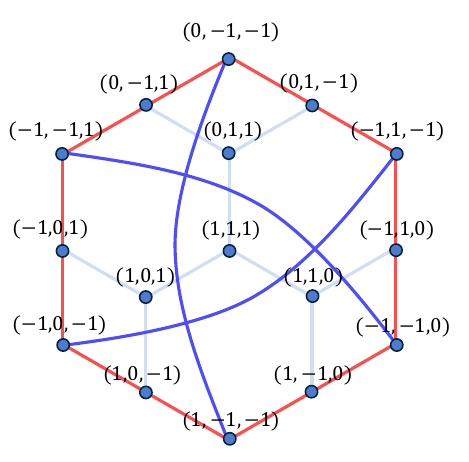}
\caption{The $K_{3,3}$-subdivision $K_{(1,-1)} \in \mathcal{K}^{\mathrm{prop}}_{1,2}(p)$ within the component $\mathcal{S}_2(p)$ of $\mathcal{G}_2(p)$}
\label{fig:k2smallcomponent}
\end{center}
\end{figure}
\end{remark}

To prove Theorem~\ref{thm-disjoint-k}, 
we prepare the following lemma.


\begin{lemma}
\label{lem-endvertex}
Let $\kappa\in\mathbb{Z}\setminus\{4\}$
and $n\in\mathbb{Z}_{>0}$ satisfying $2n+1\not\equiv 0 \pmod{p}$.
Assume that $\xi_{n,\kappa}\eta_{n,\kappa}\ne 0$. 
Suppose that $\Tnkpdist\ne\emptyset$,
and take any $(\alpha,\beta)\in \Tnkpdist$. 
\begin{itemize}
\item[$(1)$]
We have $\beta\ne 0$ and $\overline{\alpha}\ne -\alpha$.
\item[$(2)$]
Assume further that $B_{n,\kappa}(0)\ne 0$. 
Then $\alpha \neq 0$ and $\overline{\alpha} \neq 0$. 


\item[$(3)$]
Let
\[
\lambda_{n,\kappa}
\coloneq\Res\left(A_n\left(\frac{x}{2}\right),x^3+3x^2-\kappa\right)\in\mathbb{F}_p,
\]
and assume that $\lambda_{n,\kappa}\ne 0$.
Then $\alpha \neq -\beta$ and 
$\overline{\alpha} \neq -\beta$. 
Moreover, let  
$K_{(\alpha,\beta)}=\left(\begin{array}{ccc}
X_1,\!\! & X_2,\!\! & X_3 \\
Y_1,\!\! & Y_2,\!\! & Y_3
\end{array}\right)\in\mathcal{K}_{n, \kappa}^{\mathrm{dist}}(p)$.
Then for any double sign change $\sigma$
and $1\le i,j\le 3$, we have
\[
X_i, Y_j \notin \Bigl\{\sigma(X_1), \sigma(X_2), \sigma(X_3), 
\sigma(Y_1), \sigma(Y_2), \sigma(Y_3)
\Bigr\}.
\]
\end{itemize}
\end{lemma}

\begin{proof}
(1) Since $\beta$ must satisfy $A_{n}(\beta/2)=0$, 
it follows from Lemma~\ref{lem:Am properties} (2) that $\beta\ne 0$, $\overline{\alpha}=\beta^2-\alpha$ never equals $-\alpha$.
(2) Since $\alpha$ must satisfy 
$B_{n,\kappa}(\alpha)=B_{n,\kappa}(\overline{\alpha})=0$ by Corollary~\ref{cor:all equal dist} (2), 
we have $\alpha\ne 0$ and $\overline{\alpha}\ne 0$ from our assumption.
(3) Similar to the argument in (2), neither
$\alpha=-\beta$ nor
$\overline{\alpha}=-\beta$ can occur 
because both imply 
$\beta^3+3\beta^2-\kappa=0$.
Using this observation together with (1), 
one easily checks the rest of the assertion.

\end{proof}


Notice that
\[
B_{n,\kappa}(0)
=
\begin{cases}
 -\kappa+2 & \text{$n=1$}, \\
 \kappa^2-6\kappa+4   & \text{$n=2$},
\end{cases}
\qquad 
\lambda_{n,\kappa}
=
\begin{cases}
 -\kappa+2  & \text{$n=1$}, \\
 \kappa^2-5\kappa+5   & \text{$n=2$}.
\end{cases}
\]


\begin{proof}[Proof of Theorem~\ref{thm-disjoint-k}]
Consider $K_{(\alpha,\beta)} \in\Knkpprop$ 
as constructed in Theorems~\ref{thm-n=1} ($n=1$) and \ref{thm:n=2 non-planarity} ($n=2$),
which is guaranteed to be a $K_{3,3}$-subdivision for infinitely many primes $p$.
Indeed, the Chebotarev density theorem implies that the natural density
of such primes $p$ is at least $1/2$ by Theorem~\ref{thm-n=1} and \eqref{for:Chebotarev}.
The only exception is the case
$\kappa = 2$ (see Remark~\ref{rem-k=2}),
where the density is 
at least $1/4$ according to Theorem~\ref{thm:n=2 non-planarity}~(B)
with $\eta_{2,2} = -31$ and \eqref{for:Chebotarev}.
We show that, except for finitely many primes $p$,
$K_{(\alpha, \beta)}$ and $\sigma(K_{(\alpha, \beta)})$ are vertex-disjoint for each double sign change $\sigma$.
This immediately implies that 
$\sigma_1(K_{(\alpha, \beta)})$ and $\sigma_2(K_{(\alpha, \beta)})$ are vertex-disjoint for any distinct double sign changes $\sigma_1$ and $\sigma_2$.
To establish this, we employ Lemma~\ref{lem-endvertex} with $n=1,2$;
under the assumption $\kappa\ne 4$, 
only finitely many primes $p$ fail to satisfy the assumptions of Lemma~\ref{lem-endvertex}, 
and we simply exclude them.


For the sake of contradiction, 
suppose that there exists a vertex $v$ which belongs to both of $K_{(\alpha, \beta)}$ and $\sigma(K_{(\alpha, \beta)})$.
If $v$ lies on the $X_i$-$Y_j$-path for some $1 \leq i\ne  j \leq 3$ then, by Lemma~\ref{lem-double-sign-k}, 
$v$ must lie on either the $\sigma(X_i)$-$\sigma(Y_j)$-path or the $\sigma(X_j)$-$\sigma(Y_i)$-path in $\sigma(K_{(\alpha,\beta)})$, 
where both are defined by the same way as the paths in $K_{(\alpha, \beta)}$.
Since $X_i$ and $Y_j$, as well as $\sigma(X_i)$ and $\sigma(Y_j)$, are connected by the path corresponding to the alternative word consisting of $R_i$ and $R_j$, 
it follows from Lemma~\ref{lem-endvertex} (2) that either $X_i$ or $Y_j$ must also lie on the $\sigma(X_i)$-$\sigma(Y_j)$ or the $\sigma(X_j)$-$\sigma(Y_i)$-path as well.
We shall focus on the former case, as the latter can be treated similarly.
For $n=1$, there are two possibilities, that is,
$R_j(\sigma(X_i))=X_i$ or $R_j(\sigma(X_i))=Y_j$, but both are easily negated by Lemma~\ref{lem-endvertex} (1) and (2).
For $n=2$, 
we have the following six possibilities.
\begin{enumerate}
\item[(1-X)] $R_j(\sigma(X_i))=X_i$,
\item[(1-Y)] $R_j(\sigma(X_i))=Y_j$, 
\item[(2-X)] $(R_iR_j)(\sigma(X_i))=X_i$, equivalently, $R_j(\sigma(X_i))=Y_i$,
\item[(2-Y)] $(R_iR_j)(\sigma(X_i))=Y_j$, equivalently, $R_i(\sigma(Y_j))=X_j$,
\item[(3-X)] $(R_jR_iR_j)(\sigma(X_i))=X_i$, equivalently, $R_i(\sigma(Y_j))=X_i$, 
\item[(3-Y)] $(R_jR_iR_j)(\sigma(X_i))=Y_j$, equivalently, $R_i(\sigma(Y_j))=Y_j$.
\end{enumerate}
Here, in establishing the equivalences in (2-Y), (3-X) and (3-Y), 
we have used Lemma~\ref{lem-double-sign-k} 
together with the relation $(R_iR_j)^2(X_i)=Y_j$.
Cases (1-X) and (1-Y) do not hold for the same reason as in the case $n=1$.
Cases (2-X), (2-Y) and (3-Y) lead to contradictions with Lemma~\ref{lem-endvertex} (1) and (2). 
Finally, (3-X) is negated by Lemma~\ref{lem-endvertex} (3)
since it is equivalent to $\sigma(Y_j)=Y_i$. 

The remaining assertions concerning the natural densities of $p$ and $\kappa$ 
follow immediately from the same arguments used in the proof of Theorem~\ref{thm:np density}.
Actually, the desired density for the primes $p$ under consideration 
is at least
\begin{align*}
    Q(\eta_{1,\kappa};1)+Q(\eta_{2,\kappa},5;-1,1)
    -Q(\eta_{1,\kappa},\eta_{2,\kappa},5;1,-1,1) =\frac{5}{8}.
\end{align*}

\end{proof}



\begin{remark}
The four mutually vertex-disjoint $K_{3,3}$-subdivisions above are not necessarily contained in the giant component $\mathcal{C}_{\kappa}(p)$ in general.
Indeed, for $\kappa=2$ and $3$, then 
the $K_{3,3}$-subdivision $K\in \mathcal{K}^{\mathrm{prop}}_{1,\kappa}(p)$ 
obtained from Theorem~\ref{thm-n=1} (along with its three images under double sign changes) is contained in $\mathcal{S}_{2}(p)$ and $\mathcal{S}_{3}(p)$, respectively, where $\mathcal{S}_{2}(p)$ is a unique $16$-vertex component in  Remark~\ref{rem-k=2} and $\mathcal{S}_{3}(p)$ denotes a unique $72$-vertex component in $\mathcal{G}_3(p)$ corresponds to the orbit of $((1+\sqrt{5})/2, 1, 1) \in \mathcal{M}_{3}(p)$ 
under the group generated by Vieta involutions, 
provided that $(\frac{5}{p})=1$ 
(see Figure~\ref{fig:k3smallcomponent}).
On the other hand, according to a recent preprint \cite{Mart2025}, 
for $\kappa\in \mathbb{Z} \setminus \{2, 3, 4\}$, 
all the constructed $K_{3,3}$-subdivisions are expected to be contained within 
$\mathcal{C}_\kappa(p)$ of $\Gkp$. 
Furthermore, even in the cases $\kappa=2$ and $3$, 
the $K_{3,3}$-subdivision with 24 vertices constructed in Theorem~\ref{thm:n=2 non-planarity} $(n=2)$ must be contained in 
$\mathcal{C}_\kappa(p)$.
This follows from the proof of Theorem~\ref{thm-disjoint-k}, 
together with a simple observation that for $\kappa=2$ and $3$, 
any component other than $\mathcal{C}_\kappa(p)$ has at most $72$ vertices
and hence lacks sufficient space to accommodate all four mutually vertex-disjoint copies of the subdivision.
\begin{figure}[H]
\begin{center}
\includegraphics[width=0.9\linewidth, height=0.2
\textheight, keepaspectratio]{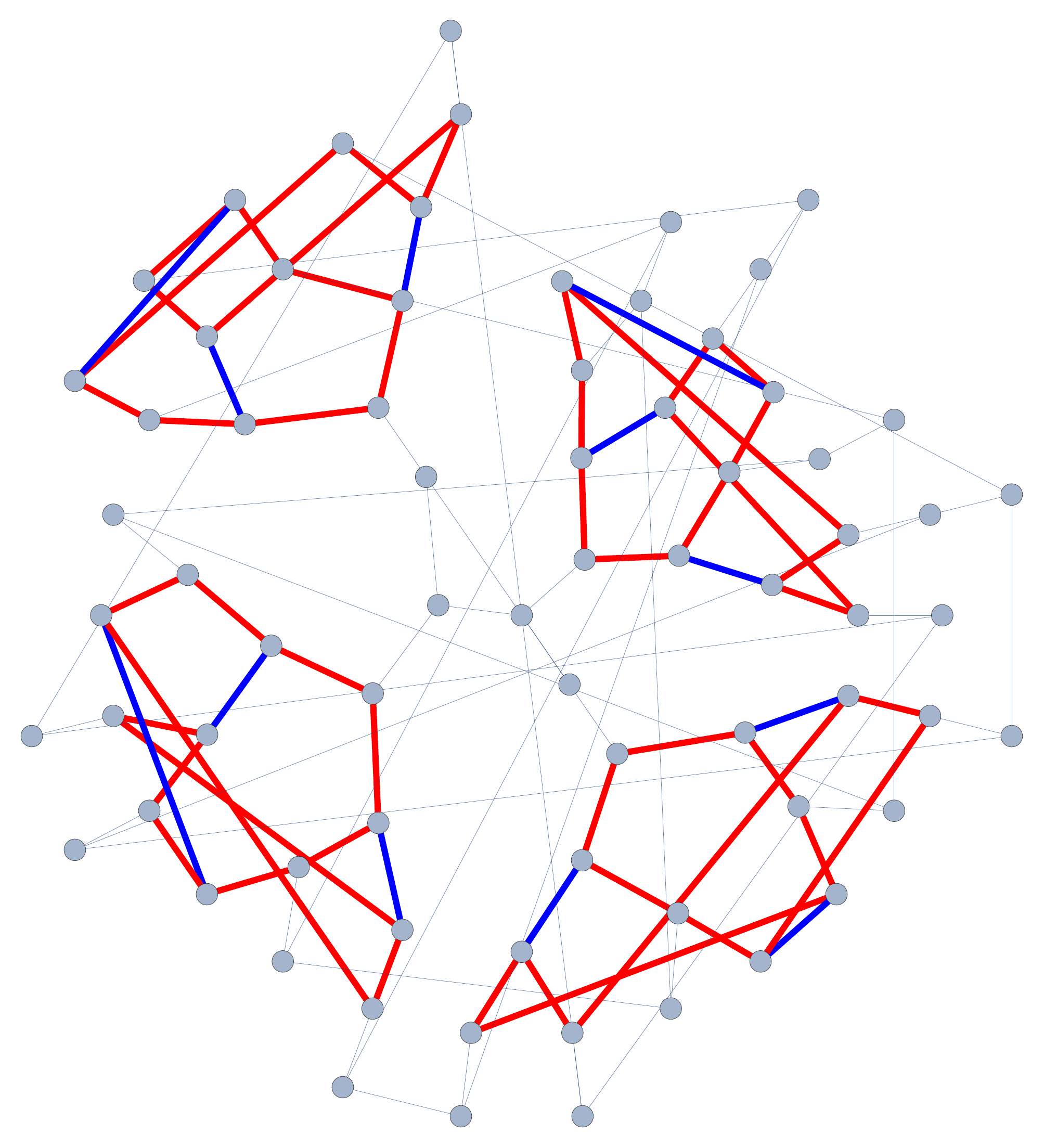}
\caption{The $72$-vertex component $\mathcal{S}_3(p)$ of  $\mathcal{G}_3(p)$ containing four mutually vertex-disjoint $K_{3,3}$-subdivisions as established in Theorem~\ref{thm-disjoint-k}}
\label{fig:k3smallcomponent}
\end{center}
\end{figure}
\end{remark}



\section{The embeddability of $\Gzp$}
\label{sect:embedded}

In this section, we discuss the embeddability of the Markoff mod $p$ graph $\Gzp$.
We say a graph $G$ is {\it toroidal} if it can be embedded into the torus.
Similarly, $G$ is said to be {\it projective-planar} if it admits an embedding into the the projective plane.

\begin{theorem}
\label{thm-proj}
Let $p>3$ be a prime. Then we have the followings.
\begin{enumerate}
\item[$(1)$] The graph $\Gzp$ is toroidal if and only if $p=7$. 
\item[$(2)$] The graph $\Gzp$ is projective-planar if and only if $p=7$. 
\end{enumerate}
\end{theorem}





\begin{proof}
We provide the proof for (1), 
as (2) follows from an analogous argument.
To this end we show that if $\Gzp$ is toroidal then $5\leq p\leq 17$.
Let $\chi$ denote the Euler characteristic of $\Gzp$, and $V$ be the number of vertices in the giant component of $\Gzp$.
Then by the same discussion to obtain (3.4) in \cite{C2024}, we have
\begin{align}
\label{eq-Euler}
    V \leq 15 \biggl(p-4-\Bigl(\frac{-1}{p}\Bigr)\biggr)+6s+2h-14\chi
\end{align}
where $s$ and $h$ denote the numbers of squares and hexagons in $\Gzp$, respectively.
Hence, the counting result for $s$ and $h$ (\cite[Lemma 2.1]{C2024}) 
combined with the inequality \eqref{eq-Euler} yield 
\begin{align}
\label{eq-Euler2}
    V \leq 
\left\{
\begin{array}{ll}
26p-90-14\chi & \text{if $p\equiv 1 \pmod{12}$},\\
26p-82-14\chi & \text{if $p\equiv 5 \pmod{12}$},\\
17p-51-14\chi & \text{if $p\equiv 7 \pmod{12}$},\\
17p-43-14\chi & \text{if $p\equiv 11 \pmod{12}$}.
\end{array}
\right.
\end{align}
Indeed, it was proved in \cite[Sections 7 and 8]{C2024} that 
\begin{align}
\label{eq-totient}
V\geq \frac{1}{2}p\phi(p+1) \geq \frac{p(p+1)}{1000\log\log (p+1)}  
\end{align}
for any $p>3$, where $\phi$ denotes the Euler totient function.
If $\Gzp$ admits an embedding into the torus, 
then we have $\chi=0$ by the Euler formula.
Combining \eqref{eq-Euler2} and \eqref{eq-totient},
it follows that $\Gzp$ can be embedded into the torus 
only if $p\leq 62440$, which can be verified by the Newton method.
In \cite{B2025}, the connectivity of $\Gzp$ is verified for $p< 10^6$, which implies that 
$\Gzp$ is connected for all $p\leq 62440$. 
Then, by \eqref{eq-number} with $\kappa=0$, 
we have $V=p^2+3(\frac{-1}{p})p$ for $p\leq 62440$, 
and the inequality \eqref{eq-Euler2} shows that 
$\Gzp$ is toroidal only if $p=5, 7, 11, 13, 17$ 
(while $\Gzp$ is non-projective-planar only if $p=7, 11, 13, 17$). 

Note that since $\mathcal{G}(7)$ is planar 
(and, of course, it can be embedded into both the torus and the projective-plane), we shall consider the remaining primes $p=5, 11, 13, 17$. 
First, for $p=11, 13, 17$, we will exhibit subdivisions of $2K_{3,3}$ (i.e. a pair of two vertex-disjoint copies of $K_{3,3}$), which is known to be a forbidden minor for toroidal graphs as well as projective-planar graphs (see, e.g.~\cite{A1981, MW2018}).
In the case $p=11$, $\mathcal{G}(11)$ contains at least four mutually vertex-disjoint $K_{3,3}$-subdivisions by Theorem~\ref{thm-disjoint-k}. 


For $p=13$ and $17$, though $p \equiv 1 \pmod{4}$, 
one can check that the construction in Theorem~\ref{thm-n=(p-1)/2} 
(which is applicable when $p \equiv 1 \pmod{4}$ for $\kappa=0$) 
does not yield a subdivision of $2K_{3,3}$. 
Instead, one can find a $2K_{3,3}$-subdivision in each case, 
as illustrated in Figures~\ref{fig:toroidal-13} and \ref{fig:toroidal-17}.

The remaining case is $p=5$,
for which there seems to be no $2K_{3,3}$-subdivisions. 
Instead, it is possible to find a subgraph 
(as shown in Figure~\ref{fig:toroidal-5-2}, e.g. \cite{MW2018}) 
which is a subdivision of another forbidden minor for toroidal and projective-planar graphs
(see Figure~\ref{fig:toroidal-5-1}).

\begin{figure}[H]
\begin{center}
\includegraphics[width=0.7\linewidth, height=0.3\textheight, keepaspectratio]{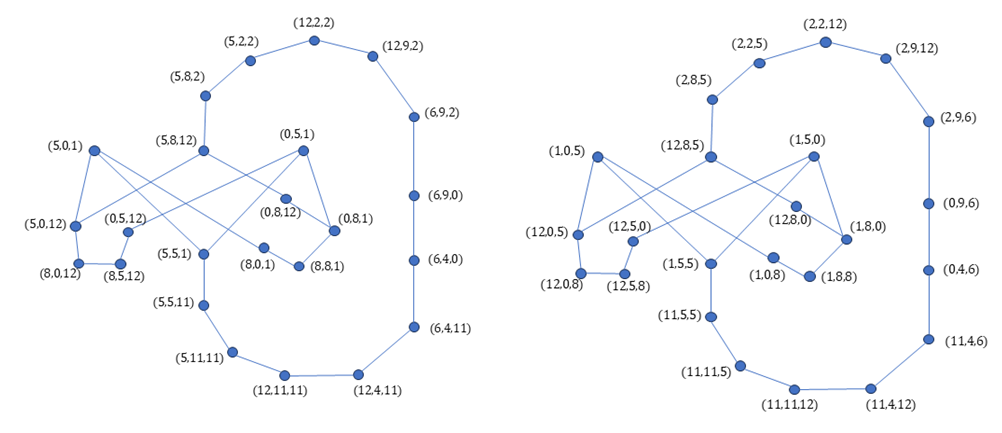}
\caption{A subdivision of $2K_{3,3}$ in $\mathcal{G}(13)$}\label{fig:toroidal-13}
\end{center}
\end{figure}


\begin{figure}[H]
\begin{center}
\includegraphics[width=0.7\linewidth, height=0.3\textheight, keepaspectratio]{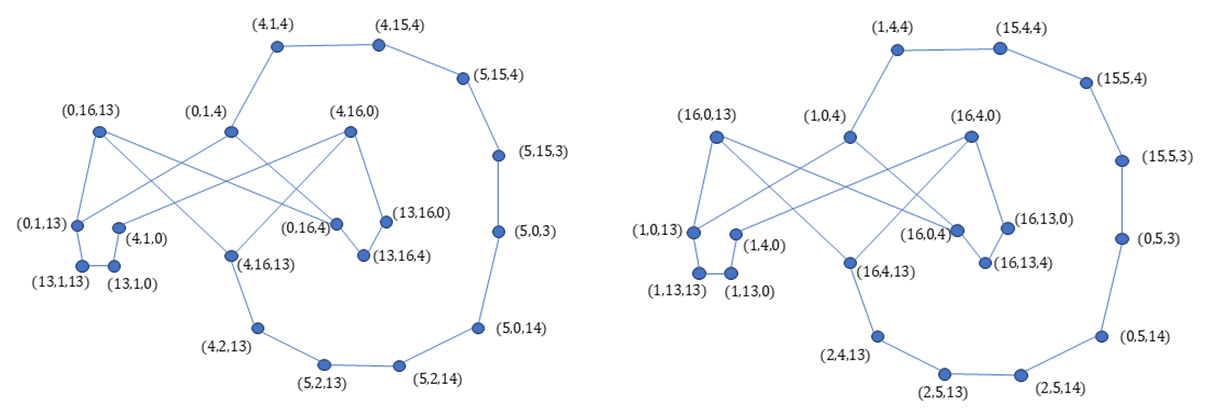}
\caption{A subdivision of $2K_{3,3}$ in $\mathcal{G}(17)$}\label{fig:toroidal-17}
\end{center}
\end{figure}


\begin{figure}[H]
  \centering
  \begin{minipage}[b]{0.4\linewidth}
    \centering
    \includegraphics[width=0.5\linewidth, height=0.45\textheight, keepaspectratio]{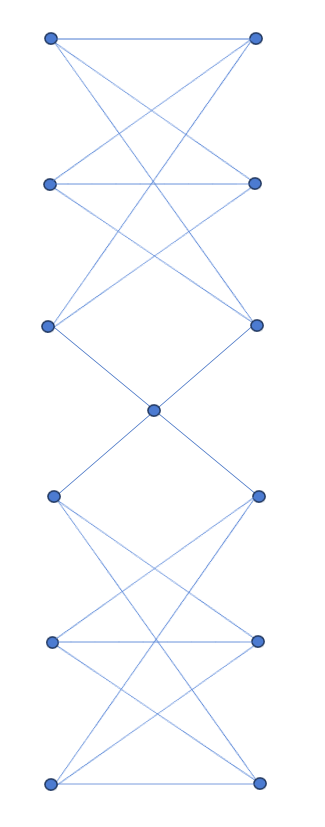}
    \caption{A forbidden minor for toroidal and projective-planar graphs}
    \label{fig:toroidal-5-1}
  \end{minipage}
  \hspace{0.05\linewidth}
  \begin{minipage}[b]{0.4\linewidth}
    \centering
    \includegraphics[width=0.7\linewidth, height=0.45\textheight, keepaspectratio]{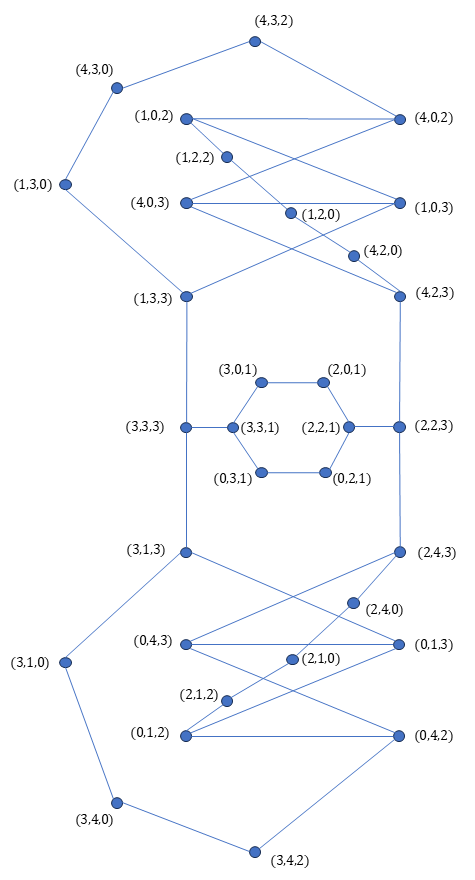}
    \caption{A minor of the left graph in $\mathcal{G}(5)$}
    \label{fig:toroidal-5-2}
  \end{minipage}
\end{figure}
\end{proof}

\section*{Acknowledgement}
The authors would like to thank Yusuke Higuchi, Yasuhiko Ikematsu, Hyungrok Jo, Masato Mimura, Kenta Ozeki, JuAe Song, Tsuyoshi Takagi and Katsuyuki Takashima for their valuable comments related to this work. 
This work was supported by JST CREST Grant Number JPMJCR2113 and JSPS KAKENHI Grant Number JP23K13007, Japan.
This work was also supported by the Institute of Mathematics for Industry, Joint Usage/Research Center in Kyushu University (Short-term Joint Research 2022a017, 2023a017, 2024a028, and 2025a037).

\end{document}